\documentclass[11pt]{article}

\usepackage[margin=1.25in]{geometry}

\usepackage{amsmath, amssymb, amsthm, authblk, rotating, hyperref, mathabx}
\usepackage[section]{placeins}
\usepackage[mathscr]{euscript}

\usepackage[sc]{mathpazo}
\usepackage[T1]{fontenc}

\usepackage{longtable}

\title{The Moonshine Module for Conway's Group}
\author{John F. R. Duncan\footnote{Email: \texttt{john.duncan@case.edu}} }
\author{Sander Mack-Crane\footnote{Email: \texttt{mack-crane@case.edu}}}
\affil{Department of Mathematics, Applied Mathematics and Statistics,\\ 
Case Western Reserve University, Cleveland, OH 44106, U.S.A.}
\date{2014 September 26}

\newcommand{\comment}[1]{}

\newcommand{\wh}[1]{\widehat{#1}}

\newcommand{\tw}{\mathrm{tw}}
\renewcommand{\a}{\mathfrak{a}}
\newcommand{\g}{\mathfrak{g}}
\renewcommand{\aa}{\hat{\mathfrak{a}}}
\newcommand{\CC}{\mathbb{C}}
\newcommand{\HH}{\mathbb{H}}
\newcommand{\LL}{\Lambda}
\newcommand{\GG}{\mathcal{G}}
\newcommand{\MM}{\mathbb{M}}

\newcommand{\QQ}{\mathbb{Q}}
\newcommand{\RR}{\mathbb{R}}

\newcommand{\vn}{V^{\natural}}
\newcommand{\vsn}{V^{s\natural}}
\newcommand{\vsnt}{V^{s\natural}_\tw}
\newcommand{\vfn}{V^{f\natural}}
\newcommand{\vfnt}{V^{f\natural}_\tw}

\newcommand{\Lo}{L(0)}
\newcommand{\Lm}{L(m)}
\newcommand{\Ln}{L(n)}
\newcommand{\vv}{\mathbf{v}}
\newcommand{\zz}{\mathfrak{z}}
\newcommand{\w}{\omega}

\newcommand{\ZZ}{\mathbb{Z}}
\newcommand{\ii}{\textbf{i}}
\newcommand{\up}{\bigtriangleup}
\newcommand{\dn}{\bigtriangledown}

\newcommand{\Aut}{\operatorname{Aut}}
\newcommand{\be}{\begin{equation}}
\newcommand{\Cliff}{\operatorname{Cliff}}
\newcommand{\Co}{\textsl{Co}}
\newcommand{\coa}{G}	
\newcommand{\coh}{\widehat{\coa}} 
\newcommand{\CM}{\operatorname{CM}}
\newcommand{\ee}{\end{equation}}
\newcommand{\End}{\operatorname{End}}
\newcommand{\Id}{\operatorname{Id}}
\newcommand{\tr}{\operatorname{tr}}
\newcommand{\str}{\operatorname{str}}

\newcommand{\lab}{\langle}
\newcommand{\rab}{\rangle}
\newcommand{\ldp}{((}
\newcommand{\rdp}{))}

\newcommand{\Spin}{\operatorname{Spin}}
\newcommand{\SO}{\operatorname{SO}}
\newcommand{\SL}{\operatorname{SL}}

\newtheorem{theorem}{Theorem}[section]
\newtheorem{proposition}[theorem]{Proposition}
\newtheorem{lemma}[theorem]{Lemma}

\newtheorem*{sthm}{Theorem}

\numberwithin{equation}{section}

\begin{document}

\maketitle

\begin{abstract}
We exhibit an action of Conway's group---the automorphism group of the Leech lattice---on a distinguished super vertex operator algebra, and we prove that the associated graded trace functions are normalized principal moduli, all having vanishing constant terms in their Fourier expansion. Thus we construct the natural analogue of the Frenkel--Lepowsky--Meurman moonshine module for Conway's group. 

The super vertex operator algebra we consider admits a natural characterization, in direct analogy with that conjectured to hold for the moonshine module vertex operator algebra. It also admits a unique canonically-twisted module, and the action of the Conway group naturally extends. We prove a special case of generalized moonshine for the Conway group, by showing that the graded trace functions arising from its action 
on the canonically-twisted module are constant in the case of Leech lattice automorphisms with fixed points, and are principal moduli for genus zero groups otherwise.
\end{abstract}

\clearpage

\tableofcontents

\section{Introduction}

Taking the upper-half plane $\HH:=\{\tau\in\CC\mid \Im(\tau)>0\}$, together with the { Poincar\'e metric} $ds^2=y^{-2}(dx^2+dy^2)$, we obtain the Poincar\'e half-plane model of the hyperbolic plane. The group of orientation preserving isometries is the quotient of $\SL_2(\RR)$ by $\{\pm I\}$, where the action is by M\"obius transformations,
\begin{gather}
	\begin{pmatrix}
	a&b\\
	c&d
	\end{pmatrix}
	\cdot \tau
	:=\frac{a\tau+b}{c\tau+d}.
\end{gather}
To any $\tau\in\HH$ we may associate a complex elliptic curve $E_{\tau}:=\CC/(\ZZ\tau+\ZZ)$, and from this point of view the modular group $\SL_2(\ZZ)$ is distinguished, as the subgroup of $\SL_2(\RR)$ whose orbits encode the isomorphism types of the curves $E_{\tau}$. That is, $E_{\tau}$ and $E_{\tau'}$ are isomorphic if and only if $\tau'=\gamma\cdot\tau$ for some $\gamma\in\SL_2(\ZZ)$. 

Subgroups of $\SL_2(\RR)$ that are commensurable with the modular group admit similar interpretations. For example, the orbits of the {Hecke congruence group} of level $N$,
\begin{gather}
	\Gamma_0(N):=
	\left\{
	\begin{pmatrix}
	a&b\\c&d
	\end{pmatrix}
	\in\SL_2(\ZZ)
	\mid c=0\pmod{N}
	\right\},
\end{gather}
correspond to isomorphism types of pairs $(E_{\tau},C)$, where $C$ is a cyclic subgroup of $E_{\tau}$ of order $N$. (Cf. e.g. \cite{MR1193029}.)

\subsection{Monstrous Moonshine}\label{sec:intro:monm}

It is a remarkable fact, aspects of which remain mysterious, that certain discrete subgroups of $\SL_2(\RR)$ in the same commensurability class as the modular group---which is to say, fairly uncomplicated groups---encode detailed knowledge of the representation theory of the largest sporadic simple group, the monster, $\MM$. Given that the monster has 
\begin{gather}
808017424794512875886459904961710757005754368000000000
\end{gather} 
elements, no non-trivial permutation representations with degree less than 
\begin{gather}
97239461142009186000
\end{gather}
and no non-trivial linear representations\footnote{Here, representation means ordinary representation, but even over a field of positive characteristic, the minimal dimension of a non-trivial representation is $196882$ (cf. \cite{MR1660407,MR1956140}).}  with dimension less than $196883$ (cf. \cite{MR0399248,MR671653} or \cite{atlas}), 
this is surprising.

The explanation of this fact relies upon the existence of a graded, infinite-dimensional representation $\vn=\bigoplus_{n\geq 0}\vn_n$ of $\MM$, such that if we define the {\em McKay--Thompson series} 
\begin{gather}\label{eqn:intro-Tm}
	T_m(\tau):=q^{-1}\sum_{n\geq 0}{\rm tr}_{\vn_n}m\, q^n
\end{gather}
for $m\in\MM$, where $q:=e^{2\pi \ii \tau}$ for $\tau\in\HH$, then the functions $T_m$ are characterized in the following way.
\begin{quote}
For each $m\in \MM$ there is a discrete group $\Gamma_m<\SL_2(\RR)$, commensurable with the modular group and having width\footnote{Say that a discrete group $\Gamma<\SL_2(\RR)$ has {\em width one at the infinite cusp} if the subgroup of upper-triangular matrices in $\Gamma$ is generated by $\pm\left(\begin{smallmatrix}1&1\\0&1\end{smallmatrix}\right)$.} one at the infinite cusp, such that $T_m$ is the unique $\Gamma_m$-invariant holomorphic function on $\HH$ satisfying $T_m(\tau)=q^{-1}+O(q)$ as $\Im(\tau)\to\infty$, and remaining bounded as $\tau$ approaches any non-infinite cusp\footnote{If $\Gamma<\SL_2(\RR)$ is commensurable with $\SL_2(\ZZ)$ then it acts naturally on $\wh{\QQ}=\QQ\cup\{\infty\}$. We define the {\em cusps} of $\Gamma$ to be the orbits of $\Gamma$ on $\wh{\QQ}$. Say that a cusp $\alpha\in\Gamma\backslash\wh{\QQ}$ is {\em non-infinite} if it does not contain $\infty$.} of $\Gamma_m$.
\end{quote}

The group $\Gamma_m$ is $\SL_2(\ZZ)$ in the case that $m=e$ is the identity element. The corresponding McKay--Thompson series $T_e$, which is the graded dimension of $\vn$ by definition, must therefore be the {\em normalized elliptic modular invariant}, 
\begin{gather}\label{eqn:intro-J}
J(\tau):=\frac{\left(1+240\sum_{n>0}\sum_{d|n}d^3q^n\right)^3}{q\prod_{n>0}(1-q^n)^{24}}-744,
\end{gather}
since this is the unique $\SL_2(\ZZ)$-invariant holomorphic function on $\HH$ satisfying $J(\tau)=q^{-1}+O(q)$ as $\Im(\tau)\to\infty$.
\begin{gather}\label{eqn:intro-vngdim}
T_e(\tau)=
q^{-1}\sum_{n\geq 0}\dim \vn_n q^{n}
	=J(\tau)
	=q^{-1}+196884q+21493760q^2+\ldots
\end{gather}
Various groups $\Gamma_0(N)$ occur as $\Gamma_m$ for elements $m\in \MM$. For example, there are two conjugacy classes of involutions in $\MM$. If $m$ belongs to the larger of these conjugacy classes, denoted $2B$ in \cite{atlas}, then $\Gamma_{2B}:=\Gamma_{m}=\Gamma_0(2)$, and 
\begin{gather}\label{eqn:intro-T2B}
	T_{2B}(\tau)
	=q^{-1}\prod_{n>0}(1-q^{2n-1})^{24}+24
	=q^{-1}+276q-2048q^2+\ldots
\end{gather}

The existence of the representation $\vn$ was conjectured by Thompson \cite{Tho_NmrlgyMonsEllModFn} following McKay's observation that $196884=1+196883$, where the significance of $196884$ is clear from (\ref{eqn:intro-vngdim}), and the significance of $1$ and $196883$ is that they are dimensions of irreducible representations of the monster. It is worth noting that at the time of McKay's observation, the monster group had not yet been proven to exist. So the monster, and therefore also its representation theory, was conjectural. The existence of an irreducible representation with dimension $196883$ was conjectured independently by Griess \cite{MR0399248} and Conway--Norton \cite{MR554399}, and the existence of the monster was ultimately proven by Griess \cite{MR671653}, via an explicit, tour de force construction of a monster-invariant (commutative but non-associative) algebra structure on the unique non-trivial $196884$-dimensional representation.

Thompson's conjecture was first confirmed, in an indirect fashion, by Atkin--Fong--Smith \cite{MR822245}, but was subsequently established in a strong sense by Frenkel--Lepowsky--Meurman \cite{FLMPNAS,FLMBerk,FLM}, who furnished a concrete construction of the {\em moonshine module} $\vn$ (cf. (\ref{eqn:intro-Tm})), together with rich, monster-invariant algebraic structure, constituting an infinite-dimensional extension of (a slight modification of) Griess' 196884-dimensional algebra. More particularly, Frenkel--Lepowsky--Meurman equipped $\vn$ with vertex operators, which had appeared originally in the dual-resonance theory of mathematical physics (cf. \cite{Sch_DulRsnThy,Man_DulRsnMdl} for reviews), and had subsequently found application (cf. \cite{MR0573075,FreKac_AffLieDualRes}) in the representation theory of affine Lie algebras.

Borcherds generalized the known constructions of vertex operators, and also derived rules for composing them in \cite{Bor_PNAS}. Using these rules he was able to demonstrate that $\vn$, together with examples arising from certain infinite-dimensional Lie algebras, admits a kind of commutative associative algebra structure---namely, {\em vertex algebra} structure---which we review in \S 2 (see [43] for a more thorough introduction). Frenkel--Lepowsky--Meurman used the fact that $\vn$ supports a representation of the Virasoro algebra to formulate the notion of {\em vertex operator algebra} in \cite{FLM}, and showed that $\MM$ is precisely the group of automorphisms of the vertex algebra structure on $\vn$ that commute with the Virasoro action. Vertex operator algebras were subsequently recognized to be ``chiral halves'' of two-dimensional conformal field theories (cf. \cite{Gaberdiel:1999mc,Gaberdiel:2005sk}), and the construction of $\vn$ by Frenkel--Lepowsky--Meurman counts as one of the earliest examples of an orbifold conformal field theory (cf. \cite{MR818423,MR851703,MR968697}).

The characterization of the McKay--Thompson series $T_m$ quoted above is the main content of the {\em monstrous moonshine conjectures}, formulated by Conway--Norton in \cite{MR554399} (see also \cite{Tho_FinGpsModFns}), and solved by Borcherds in \cite{borcherds_monstrous}. It is often referred to as the {\em genus zero property} of monstrous moonshine, because the existence of a function $T_m$ satisfying the given conditions implies that $\Gamma_m$ has {\em genus zero}, in the sense that the orbit space $\Gamma_m\backslash \wh{\HH}$ is isomorphic as a Riemann surface\footnote{Cf. e.g. \cite{Shi_IntThyAutFns} for the Riemann surface structure on $\Gamma\backslash\wh{\HH}$.} to the Riemann sphere $\wh{\CC}=\CC\cup\{\infty\}$, where $\wh{\HH}:=\HH\cup\QQ\cup\{\infty\}$. Indeed, the function $T_m$ witnesses this, as it induces an embedding $\Gamma_m\backslash\HH\to\CC$ which extends uniquely to an isomorphism $\Gamma_m\backslash\wh{\HH}\to \wh{\CC}$.

Conversely, if $\Gamma<\SL_2(\RR)$ has genus zero in the above sense, then there is an isomorphism of Riemann surfaces $\Gamma\backslash\wh{\HH}\to \wh{\CC}$, and the composition $\HH\to\wh{\HH}\to\Gamma\backslash\wh{\HH}\to \wh{\CC}$ maps $\HH$ to $\CC$, thereby defining a $\Gamma$-invariant holomorphic function $T_\Gamma$ on $\HH$. Call a $\Gamma$-invariant holomorphic function $T_\Gamma:\HH\to \CC$ a {\em principal modulus} for $\Gamma$, if it arises in this way from an isomorphism $\Gamma\backslash\widehat{\HH}\to\widehat{\CC}$. If $\Gamma$ has width one at the infinite cusp, then, after applying an automorphism of $\wh{\CC}$ if necessary, we have 
\begin{gather}\label{eqn:intro-pmcrit}
T_\Gamma(\tau)=q^{-1}+ O(q)
\end{gather} 
as $\Im(\tau)\to\infty$, and no poles at any non-infinite cusps of $\Gamma$. Such a function $T_\Gamma$---call it a {\em normalized principal modulus} for $\Gamma$---is unique because the difference between any two defines a holomorphic function on the compact Riemann surface $\Gamma\backslash\widehat{\HH}$ which vanishes at the infinite cusp by force of (\ref{eqn:intro-pmcrit}). The only holomorphic functions on a compact Riemann surface are constants (cf. e.g. \cite{MR1013999}) and hence this difference vanishes identically.

So knowledge of the McKay--Thompson series $T_m$ is equivalent to knowledge of the discrete groups $\Gamma_m<\SL_2(\RR)$, according to the characterization furnished by monstrous moonshine. This explains the claim we made above, that subgroups of $\SL_2(\RR)$ ``know'' about the representation theory of the monster. For, according to the definition (\ref{eqn:intro-Tm}) of $T_m$, we can compute the graded trace of a monster element $m\in \MM$ on the moonshine module $\vn$, as soon as we know the group $\Gamma_m$. In particular, we can compute the traces of monster elements on its infinite-dimensional representation $\vn$, without doing any computations in the monster itself.

As has been mentioned above, a concrete realization of the the McKay--Thompson series $T_m$ is furnished by the Frenkel--Lepowsky--Meurman construction of the moonshine module $\vn$. Their method was inspired in part by the original construction of the monster due to Griess \cite{MR605419,MR671653}, and takes Leech's lattice \cite{Lee_SphPkgHgrSpc,Lee_SphPkgs} as a starting point. Conway has proven \cite{Con_ChrLeeLat} that the Leech lattice $\LL$ is the unique up to isomorphism even positive-definite lattice such that
\begin{itemize}
\item the rank of $\LL$ is $24$, 
\item $\LL$ is self-dual, and
\item $\lab \lambda,\lambda\rab\neq 2$ for any $\lambda\in \LL$.
\end{itemize}
Conway also studied the automorphism group of $\LL$ and discovered three new sporadic simple groups in the process \cite{MR0237634,MR0248216}. We set $\Co_0:=\Aut(\LL)$ and call it Conway's group. The largest of Conway's sporadic simple groups is the quotient 
\begin{gather}\label{eqn:intro-Co1}
\Co_1:=\Aut(\LL)/\{\pm \Id\}
\end{gather} 
of $\Co_0$ by its center.

\subsection{Conway Moonshine}\label{sec:intro:conm}

In their paper \cite{MR554399}, Conway--Norton also described an assignment of genus zero groups $\Gamma_g<\SL_2(\RR)$, to elements $g$ of the Conway group, $\Co_0$. Their prescription is very concrete and may be described as follows. If $g\in \Co_0=\Aut(\LL)$ acts on $\LL\otimes_{\ZZ}\CC$ with eigenvalues $\{\varepsilon_i\}_{i=1}^{24}$, then $\Gamma_g$ is the invariance group of the holomorphic function
\begin{gather}\label{eqn:intro-defntg}
	t_g(\tau):=q^{-1}\prod_{n>0}\prod_{i=1}^{24}\left(1-\varepsilon_iq^{2n-1}\right)
\end{gather}
on $\HH$.
Observe that for $g=e$ the identity element of $\Co_0$, the function $t_e$ almost coincides with the monstrous McKay--Thompson series $T_{2B}$ of (\ref{eqn:intro-T2B}): the latter has vanishing constant term, whereas $t_e(\tau)=q^{-1}-24+O(q)$. 
That the invariance groups of the $t_g$ actually are genus zero subgroups of $\SL_2(\RR)$, and that the $t_g$ are principal moduli, was verified in part in \cite{MR554399}, and in full in \cite{MR628715}. (One may see \cite{MR780666} or Table \ref{tab:mmdata-Tsg} in \S\ref{sec:mmdata} of this paper for an explicit description of all the groups $\Gamma_g$ for $g\in \Co_0$.) 

So discrete subgroups of isometries of the hyperbolic plane also know a lot about the representation theory of Conway's group $\Co_0$, but to fully justify this statement we should construct the appropriate analogue of $\vn$; that is, a graded infinite-dimensional $\Co_0$-module whose graded trace functions recover the $t_g$.

Note at this point that the functions $t_g$ are not distinguished to quite the extent that the monstrous McKay--Thompson series $T_m$ are, for $m\in \MM$, because for $g\in \Co_0$ we have
\begin{gather}
	t_g(\tau)=q^{-1}-\chi_g+O(q),
\end{gather}
where $\chi_g=\sum_{i=1}^{24} \varepsilon_i$ is the trace of $g$ attached to its action on $\LL\otimes_{\ZZ}\CC$. In particular, $t_g$ generally has a non-vanishing constant term, and does not satisfy the criterion (\ref{eqn:intro-pmcrit}) defining normalized principal moduli. So for $g\in \Co_0$ let us define
\begin{gather}\label{eqn:intro-defnTg}
	T^s_g(\tau):=t_g(\tau/2)+\chi_g=q^{-1/2}\prod_{n>0}\prod_{i=1}^{24}\left(1-\varepsilon_iq^{n-1/2}\right)+\chi_g,
\end{gather}
so that $T^s_g(2\tau)$ is the unique normalized principal modulus attached to the genus zero group $\Gamma_g$.

The purpose of this article is to construct a $\frac12\ZZ$-graded infinite-dimensional $\Co_0$-module, $\vsn=\bigoplus_{n\geq 0} \vsn_{n/2}$, 
such that the normalized principal moduli $T^s_g$, for $g\in \Co_0$, are obtained via an analogue of (\ref{eqn:intro-Tm}). That is, we will construct the natural analogue of the moonshine module $\vn$ for the Conway group $\Co_0$.

In fact we will do better than this, by establishing a characterization of the algebraic structure underlying $\vsn$. 
In \cite{FLM} it is conjectured that $\vn$ is the unique vertex operator algebra such that 
\begin{itemize}
\item the central charge of $\vn$ is $24$, 
\item $\vn$ is self-dual, and
\item $\deg(v)\neq 1$ for any non-zero $v\in \vn$.
\end{itemize}
Just as vertex operator algebra structure is a crucial feature of the moonshine module $\vn$, super vertex operator algebra furnishes the correct framework for understanding $\vsn$ (and also motivates the rescaling of $\tau$ in (\ref{eqn:intro-defnTg})).
In \S\ref{sec:moon:uniq} we prove the following result.
\begin{sthm}[{\bf \ref{thm:moon:uniq-uniq}}]
There is a unique up to isomorphism $C_2$-cofinite rational super vertex operator algebra of CFT type $\vsn$ such that
\begin{itemize}
\item the central charge of $\vsn$ is $12$, 
\item $\vsn$ is self-dual, and
\item $\deg(v)\neq \frac12$ for any non-zero $v\in \vsn$.
\end{itemize}
\end{sthm}
(We refer to \S\ref{sec:va:funds} for explanations of the technical terms appearing in the statement of Theorem \ref{thm:moon:uniq-uniq}. We say that a super vertex operator algebra is self-dual if it is rational, irreducible as a module for itself, and if it is its only irreducible module up to isomorphism. We write $\deg(v)=n$ in case $L(0)v=nv$.)

As explained in detail in \cite{FLM}, the conjectural characterization of $\vn$ above puts the monster, and $\vn$, at the top tier of a three-tier tower 
\begin{gather}
	\begin{split}\label{eqn:intro-FLMtower}
		\MM\curvearrowright\vn&\text{\quad---\quad vertex operator algebras}\\
		\Co_0\curvearrowright\LL&\text{\quad---\quad even positive-definite lattices}\\
		M_{24}\curvearrowright\mathcal{G}&\text{\quad---\quad doubly-even linear binary codes}
	\end{split}
\end{gather}
involving three sporadic groups and three distinguished structures, arising in vertex algebra, lattice theory, and coding theory, respectively.

An important structural feature of this tower is the evident parallel between the conjectural characterization of $\vn$, and Conway's characterization of the Leech lattice quoted earlier. In (\ref{eqn:intro-FLMtower}) we write $\mathcal{G}$ for the {\em extended binary Golay code}, introduced (essentially) by Golay in \cite{Gol_NtsDgtCdg}, which is the unique (cf. e.g. \cite{MR1662447,MR1667939}) doubly-even linear binary code such that 
\begin{itemize}
\item the length of $\mathcal{G}$ is $24$, 
\item $\mathcal{G}$ is self-dual, and
\item $wt(C)\neq 4$ for any word $C\in\mathcal{G}$.
\end{itemize}
The automorphism group of $\mathcal{G}$ is the largest sporadic simple group discovered by Mathieu \cite{Mat_1861,Mat_1873}, denoted here by $M_{24}$.

Theorem \ref{thm:moon:uniq-uniq} now suggests that $\vsn$ may serve as a replacement for $\LL$ in the tower (\ref{eqn:intro-FLMtower}). With the recent development of Mathieu moonshine (see \cite{Eguchi2010} for an account of the original observation, and \cite{MR2985326} for a review), one may speculate about the existence of a vertex algebraic replacement for $\mathcal{G}$, and an alternative tower to (\ref{eqn:intro-FLMtower}), with tiers corresponding to moonshine in three forms.
\begin{gather}
	\begin{split}\label{eqn:intro-moonshinetower}
		\MM\curvearrowright\vn&\text{\quad---\quad monstrous moonshine}\\
		\Co_0\curvearrowright\vsn&\text{\quad---\quad Conway moonshine}\\
		M_{24}\curvearrowright\; ??&\text{\quad---\quad Mathieu moonshine}
	\end{split}
\end{gather}
Note that a natural analogue of the genus zero property of monstrous moonshine, and the moonshine for Conway's group considered here, has been obtained for Mathieu moonshine in \cite{Cheng2011}.

We give an explicit construction for $\vsn$ in \S\ref{sec:moon:cnstn}. The proof of the characterization result, Theorem \ref{thm:moon:uniq-uniq}, demonstrates that the even part of $\vsn$ is isomorphic to the lattice vertex algebra of type $D_{12}$, which, according to the boson-fermion correspondence (cf. \cite{MR643037,MR1284796}), is the even part of the Clifford module super vertex operator algebra $A(\a)$ attached to a $24$-dimensional orthogonal space, $\a$. From this point of view the connection to the Conway group is reasonably transparent, for we may identify $\a$ with the space $\LL\otimes_{\ZZ}\CC$ enveloping the Leech lattice. We realize $\vsn$ by performing a $\ZZ/2$-orbifold of the Clifford module super vertex operator algebra $A(\a)$. Our method also produces an explicit construction of the unique (up to isomorphism) canonically-twisted $\vsn$-module, which we denote $\vsn_\tw$, and which is also naturally a $\Co_0$-module. (Canonically-twisted modules for super vertex algebras are reviewed in \S\ref{sec:va:funds}, and the Clifford module super vertex operator algebra construction is reviewed in \S\ref{sec:va:cliff}.)

Our main results appear in \S\ref{sec:moon:mtseries}, where we consider the graded super trace functions 
\begin{gather}\label{eqn:intro-TgTgtw}
	T^s_g(\tau):=q^{-1/2}\sum_{n\geq 0}{\rm str}_{\vsn_{n/2}}g\, q^{n/2},\\
	T^s_{g,\tw}(\tau):=q^{-1/2}\sum_{n\geq 0}{\rm str}_{\vsn_{\tw,n+1/2}}g\, q^{n+1/2},
\end{gather}
arising naturally from the actions of $\Co_0$ on $\vsn$ and $\vsn_\tw$. (Necessary facts about the Conway group are reviewed in \S\ref{sec:gps:conway}.) 
\begin{sthm}[{\bf \ref{thm:moon:mtseries-mainthm1}}]
Let $g\in \Co_0$. Then $T^s_g$ is the normalized principal modulus for a genus zero subgroup of $\SL_2(\RR)$.
\end{sthm}

The statement that the $T^s_g$ are normalized principal moduli is a direct analogue, for the Conway group, of the monstrous moonshine conjectures, formulated by Conway--Norton in \cite{MR554399}. The moonshine conjectures were broadly expanded by Norton in \cite{MR933359,generalized_moonshine}, to an association of functions $T_{(m,m')}(\tau)$ to pairs $(m,m')$ of commuting elements in the monster. Norton's {\em generalized moonshine conjectures} state, among other things (see \cite{MR1877765} for a revised formulation), that $T_{(m,m')}$ should be a principal modulus for a genus zero group $\Gamma_{(m,m')}$, or a constant function, for every commuting pair $m,m'\in\MM$. In terms of vertex operator algebra theory, the functions $T_{(m,m')}$ should be defined by traces on twisted modules for $\vn$ (cf. \cite{Dong2000}).

The argument used to prove Theorem \ref{thm:moon:mtseries-mainthm1} also establishes the following result, which we may regard as confirming a special case of generalized moonshine for the Conway group.
\begin{sthm}[{\bf \ref{thm:moon:mtseries-mainthm2}}]
Let $g\in \Co_0$. Then $T^s_{g,\tw}$ is constant, with constant value $-\chi_g$, when $g$ has a fixed point in its action on the Leech lattice. If $g$ has no fixed points then $T^s_{g,\tw}$ is a principal modulus for a genus zero subgroup of $\SL_2(\RR)$.
\end{sthm}
The problem of precisely formulating, and proving, generalized moonshine for the Conway group is an important direction for future work.

Generalized moonshine for the monster remains unproven in general, although a number of special cases have been established, by Dong--Li--Mason in \cite{MR1372716,Dong2000}, Ivanov--Tuite in \cite{MR1915258,MR1915259}, and H\"ohn in \cite{Hoe_GenMnsBbyMns}. The most general results on generalized moonshine are due to Carnahan \cite{MR2728485,MR2904095,Carnahan:2012gx,Car_MnsLieAlg}.

The present paper is closely related to earlier work \cite{Dun_VACo}, in which a vertex algebraic construction of Conway's sporadic simple group $\Co_1$ (cf. (\ref{eqn:intro-Co1})) was obtained. In \cite{Dun_VACo} a super vertex operator algebra $\vfn=\bigoplus_{n\geq 0}\vfn_{n/2}$ is defined (over the real numbers; we work here over $\CC$), and it is proven that $\vfn$ admits an $N=1$ structure---a certain super extension of the Virasoro action, cf. \S\ref{sec:va:funds}---for which the automorphism group is precisely $\Co_1$. In fact, the super vertex operator algebras $\vsn$ and $\vfn$ (when defined over $\CC$) are isomorphic, and the essential difference between this paper and \cite{Dun_VACo} is that we consider a different action of the Conway group: in the case of $\vsn$ the action is faithful, but for $\vfn$ the central subgroup $\{\pm\Id\}$ acts trivially, and thus one obtains an action of the simple group $\Co_1$ on $\vfn$. 

Both this work and \cite{Dun_VACo} rest upon the important antecedents \cite{MR1390654,FLMBerk}. In \S15 of \cite{FLMBerk}, the construction of the super vertex operator algebra $\vfn$ is described for the first time, and it is conjectured that the simple Conway group $\Co_1$ should act on $\vfn$ as automorphisms. Later, in \S5 of \cite{MR1390654}, a lattice super vertex operator algebra is identified, which turns out to be isomorphic to both $\vfn$ and $\vsn$, and it is explained that both $\Co_1$ and $\Co_0$ can act faithfully on this object. The construction of $\vsn$ given in \S\ref{sec:moon:cnstn} differs significantly from that of $\vfn$ described in \cite{FLMBerk}, but is closely connected, via the boson-fermion correspondence, to the description given in \cite{MR1390654}.

Although the graded trace functions attached to the action of $\Co_1$ on $\vfn$ are computed explicitly in \cite{Dun_VACo}, their modular properties are not considered in detail. From the point of view of moonshine, the trace functions arising from $\vsn$ are better: the functions 
\begin{gather}\label{eqn:intro-Tfg}
	T^f_g(\tau):=q^{-1/2}\sum_{n\geq 0}\str_{\vfn_{n/2}}g\,q^{n/2},
\end{gather}
defined for $g\in \Co_1$,
are generally not principal moduli, even though they satisfy the normalization condition (\ref{eqn:intro-pmcrit}). Nonetheless, the $\Co_1$-module structure on $\vfn$ plays an important role in the present paper. A crucial step in our proof of Theorems \ref{thm:moon:mtseries-mainthm1} and \ref{thm:moon:mtseries-mainthm2} is the verification of a $\Co_0$-family of eta-product identities (\ref{eqn:moon:mtseries-etaid}), which we establish in Lemma \ref{lem:moon:mtseries-etaid}. We are able to prove these in a uniform manner by utilizing the unique (cf. Proposition \ref{prop:moon:cnstn-cotons}) $\Co_1$-invariant $N=1$ structure on $\vfn$.

Another work of particular relevance to the moonshine for Conway's group we consider here is \cite{DunFre_RSMG}, in which the McKay--Thompson series (\ref{eqn:intro-Tm}) of monstrous moonshine are characterized, following earlier work \cite{ConMcKSebDiscGpsM}, in terms of certain regularized Poincar\'e series, called {Rademacher sums}. The theorems in \S6 of \cite{DunFre_RSMG} imply that a discrete group $\Gamma<\SL_2(\RR)$, commensurable with $\SL_2(\ZZ)$ and having width one at the infinite cusp, has genus zero, if and only if the associated Rademacher sum
\begin{gather}\label{eqn:intro:conm-RGamma}
	R_{\Gamma}(\tau):=q^{-1}+\lim_{K\to \infty}
	\sum_{\gamma\in\Gamma_{\infty}\backslash\Gamma_{<K}^\times}
	\left(e^{-2\pi \ii\gamma\tau}-e^{-2\pi \ii\gamma\infty}\right)
\end{gather}
is a principal modulus for $\Gamma$. (In (\ref{eqn:intro:conm-RGamma}) we write $\Gamma_\infty$ for the subgroup of upper-triangular matrices in $\Gamma$, and $\Gamma_\infty\backslash\Gamma_{<K}^\times$ denotes the set of non-trivial cosets for $\Gamma_\infty$ in $\Gamma$ such that if a representative $\gamma$ is re-scaled to a matrix $\left(\begin{smallmatrix} a&b\\c&d\end{smallmatrix}\right)$ with integer entries and $c>0$, then $c<K$ and $-K^2<d<K^2$. In case $\Gamma=\SL_2(\ZZ)$ the right-hand side of (\ref{eqn:intro:conm-RGamma}) is exactly the expression given for $J(\tau)+12$ by Rademacher in \cite{Rad_FuncEqnModInv}.) 

So in particular, the results of \cite{DunFre_RSMG} imply that all the McKay--Thompson series $T^s_g$ of Conway moonshine, attached to the Conway group via its action on $\vsn$, can be realized\footnote{Strictly speaking, the Rademacher sum $R_{\Gamma_g}(\tau)$ generally differs from $T^s_g(2\tau)$ by an additive constant, and a similar statement is true for the functions of monstrous moonshine. The normalized Rademacher sums, defined in \S4 of \cite{DunFre_RSMG}, have vanishing constant terms, and thus recover the $T_m$ and $T^s_g$ precisely.} as Rademacher sums. Thus we obtain a uniform construction of the $T^s_g$, as the (normalized) Rademacher sums attached to their invariance groups $\Gamma_g$. 

The formulation of a characterization of the $T^s_g$, in analogy with that given for the $T_m$ in Theorem 6.5.1 of \cite{DunFre_RSMG}, is another important problem for future work.

\subsection{Mathieu Moonshine}\label{sec:intro:matm}

The significance of the super vertex operator algebra $\vsn$ is further demonstrated by recent developments in Mathieu moonshine---mentioned above, in connection with the tower (\ref{eqn:intro-moonshinetower})---which features an assignment of weak Jacobi forms\footnote{The standard reference for the theory of Jacobi forms is \cite{eichler_zagier}.} 
of weight zero and index one, to conjugacy classes in the sporadic simple Mathieu group, $M_{24}$. 

In forthcoming work \cite{vacoeg} we demonstrate how the canonically-twisted $\vsn$-module $\vsn_\tw$ may be used to attach weak Jacobi forms of weight zero and index one to conjugacy classes in $\Co_0$ that fix a four-space in $\LL\otimes_{\ZZ}\CC$. The group $M_{24}$ is a subgroup of $\Co_0$, and our assignment recovers many (but not all) of the weak Jacobi forms attached to $M_{24}$ by Mathieu moonshine. Mathieu moonshine is identified as a special case of $23$ similar moonshine phenomena, collectively known as umbral moonshine, in \cite{UM,MUM}, and our construction generalizes naturally, so as to attach weak Jacobi forms of index greater than one to suitable elements of $\Co_0$. Several of the higher index Jacobi forms of umbral moonshine arise in this way, from the faithful action of $\Co_0$ on $\vsn_\tw$.

The weak Jacobi forms of Mathieu moonshine may be replaced with mock modular forms of weight $1/2$, by utilizing the irreducible unitary characters of the small $N=4$ superconformal algebra. (See \cite{MR2985326} for a review of this, including an introductory discussion of mock modular forms.) One of the main conjectures of Mathieu moonshine is that these mock modular forms---namely, those prescribed in \cite{MR2793423,Eguchi2010a,Gaberdiel2010a,Gaberdiel2010}, and defined in a uniform way, via Rademacher sums, in \cite{Cheng2011}---are the graded traces attached to the action of $M_{24}$ on some graded infinite-dimensional $M_{24}$-module. Despite the work of Gannon \cite{Gannon:2012ck}, proving the existence of a such an $M_{24}$-module,
an explicit construction of the Mathieu moonshine module is still lacking. The result of \cite{vacoeg} stated above, that many of the weak Jacobi forms of Mathieu moonshine may be recovered from an action of $M_{24}$ on $\vsn_\tw$, demonstrates that $\vsn$ may serve as an important tool in the construction of this moonshine module for $M_{24}$.

Strong evidence in support of the idea that $\vsn$ can play a role in the construction of modules for Mathieu moonshine, and umbral moonshine more generally, is given in \cite{2014arXiv1406.5502C}, where $\vsn_\tw$ is used to attach weak Jacobi forms of weight zero and index two to all of the elements of the sporadic simple Mathieu groups $M_{23}$ and $M_{22}$ (characterized as point stabilizers in $M_{24}$ and $M_{23}$, respectively). Further, it is shown that the representation theory of the $N=2$ and small $N=4$ superconformal algebras (cf. \cite{Eguchi1988a}) naturally leads to assignments of (vector-valued) mock modular forms to the elements of these groups. Thus the first examples of explicitly realized modules underlying moonshine phenomena relating mock modular forms to sporadic simple groups are obtained via $\vsn$ in \cite{2014arXiv1406.5502C}.

Aside from the interesting connections to umbral moonshine, the main result of the forthcoming work \cite{vacoeg} is the assignment of a weak Jacobi form of weight zero and index one to any symplectic derived autoequivalence of a projective complex K3 surface that fixes a stability condition in the distinguished space defined by Bridgeland in \cite{MR2376815}. Conjecturally, the data of such a stability condition is equivalent to the physical notion of a supersymmetric non-linear sigma model on the corresponding K3 surface. 
(See \cite{MR1479699} for a detailed discussion of the moduli space of K3 sigma models, \cite{MR2955931} for a concise treatment, and \cite{2013arXiv1309.6528H} for the relationship with stability conditions.) 
As demonstrated by Witten in \cite{MR970278}, a supersymmetric non-linear sigma model defines a weak Jacobi form, called the {\em elliptic genus} of the sigma model in question, and it turns out that the Jacobi form one obtains in the case of a(ny) K3 sigma model is precisely that arising from the identity element of $\Co_0$ in the construction of \cite{vacoeg}. 

More generally, one expects, on physical grounds (cf. \cite{MR2955931}), to obtain a weak Jacobi form (with level) from any supersymmetry-preserving automorphism of a non-linear sigma model---call it a {\em twined elliptic genus}---and it is shown in \cite{MR2955931} that the automorphism groups of K3 sigma models are the subgroups of $\Co_0$ that fix four-spaces\footnote{This is a quantum analogue of the celebrated result of Mukai \cite{Mukai}, that the finite groups of symplectic automorphisms of a K3 surface are the subgroups of the sporadic simple Mathieu group $M_{23}$ that have at least five orbits in their action on $24$ points.} in $\LL\otimes_{\ZZ}\CC$. In general it is hard to compute twined elliptic genera, for the the Hilbert spaces underlying non-linear sigma models can, so far, only be constructed for certain special examples. Nonetheless, we find that the construction of \cite{vacoeg} agrees precisely with the explicit computations of \cite{MR2955931,Gaberdiel:2012um,2014arXiv1403.2410V}, which account for about half the conjugacy classes of $\Co_0$ that fix a four-space in $\LL\otimes_{\ZZ}\CC$. Thus the main result of \cite{vacoeg} indicates that $\vsn$ may serve as a kind of universal object for understanding the twined elliptic genera of K3 sigma models. It may develop that $\vsn$ can shed light on more subtle structural aspects of K3 sigma models also.

The discussion here indicates that $\vsn$ plays an important role in Mathieu moonshine, and umbral moonshine more generally. On the other hand, that Conway moonshine and monstrous moonshine are closely related is evident from the discussions in \S\S\ref{sec:intro:monm},\ref{sec:intro:conm}. Thus the results of this paper furnish further evidence---see also \cite{2014arXiv1403.3712O}, and the introduction to \cite{UM}---that monstrous moonshine and umbral moonshine are related in a deep way, possibly having a common origin.

\subsection{Organization}

The organization of the paper is as follows. We review facts from vertex algebra theory in \S\ref{sec:va}. Basic notions are recalled in \S\ref{sec:va:funds}, invariant bilinear forms on super vertex algebras are discussed in \S\ref{sec:va:ibfs}, and the Clifford module super vertex operator algebra construction is reviewed in \S\ref{sec:va:cliff}. Spin groups act naturally on Clifford module super vertex operator algebras, and we review this in detail in \S\ref{sec:gps:spin}. All necessary facts about the Conway group are explained in \S\ref{sec:gps:conway}. The main results of the paper appear in \S\ref{sec:moon}, which features an explicit construction of $\vsn$ in \S\ref{sec:moon:cnstn}, the characterization of $\vsn$ in \S\ref{sec:moon:uniq}, and the analysis of its trace functions in \S\ref{sec:moon:mtseries}. The paper concludes with tables in \S\ref{sec:mmdata}, one for the $T^s_g$, and one for the $T^s_{g,\tw}$, which can be used to facilitate explicit computations.

\section{Vertex Algebra}\label{sec:va}

In this section we recall some preliminary facts from vertex algebra. In addition to the specific references that follow, we refer to the texts \cite{MR2082709,MR1651389,MR2023933} for more background on vertex algebras, vertex operator algebras, and the various kinds of modules over these objects.

\subsection{Fundamental Notions}\label{sec:va:funds}

A \emph{super vertex algebra} is a super vector space $V=V_{\bar{0}}\oplus V_{\bar{1}}$ equipped with a \emph{vacuum vector} $\mathbf{1} \in V_{\bar{0}}$, a linear operator $T:V\to V$, and a linear map
\begin{align}
\begin{split}\label{eqn:va:funds-vopcorr}
V & \to\End(V)[[z^{\pm1}]] \\
a & \mapsto Y(a,z)=\sum_{n\in\ZZ} a_{(n)} z^{-n-1}
\end{split}
\end{align}
associating to each $a\in V$ a \emph{vertex operator} $Y(a,z)$, which satisfy the following axioms for any $a,b, c\in V$:
\begin{enumerate}
\item\label{itm:va:funds-vaaxmer} $Y(a,z)b\in V\ldp z\rdp$ and if $a\in V_{\bar{0}}$ (resp. $a\in V_{\bar{1}}$) then $a_{(n)}$ is an even (resp. odd) operator for all $n$;
\item\label{itm:va:funds-vaaxid} $Y(\mathbf{1},z)=\Id_V$ and $Y(a,z)\mathbf{1}\in a+zV[[z]]$;
\item\label{itm:va:funds-vaaxtrans} $[T,Y(a,z)]=\partial_z Y(a,z)$, $T\mathbf{1}=0$, and $T$ is an even operator; and
\item\label{itm:va:funds-vaaxassoc} if $a\in V_{p(a)}$ and $b\in V_{p(b)}$ are $\ZZ/2$ homogenous, there exists an element
$$f\in V[[z,w]][z^{-1},w^{-1},(z-w)^{-1}]$$
depending on $a$, $b$, and $c$, such that
$$Y(a,z)Y(b,w)c,\quad (-1)^{p(a)p(b)}Y(b,w)Y(a,z)c,\quad\text{and}\quad Y(Y(a,z-w)b,w)c$$
are the expansions of $f$ in $V\ldp z\rdp\ldp w\rdp$, $V\ldp w\rdp\ldp z\rdp$, and $V\ldp w\rdp\ldp z-w\rdp$, respectively.
\end{enumerate}
In items \ref{itm:va:funds-vaaxmer} and \ref{itm:va:funds-vaaxassoc} above we write $V\ldp z\rdp$ for the vector space $V\ldp z\rdp=V[[z]][z^{-1}]$ whose elements are formal Laurent series in $z$ with coefficients in $V$. Note that $\CC\ldp z\rdp$ is naturally a field, and if $V$ is a vector space over $\CC$ then $V\ldp z\rdp$ is naturally a vector space over $\CC\ldp z\rdp$.

A \emph{module} over a super vertex algebra $V$ is a super vector space $M=M_{\bar{0}}\oplus M_{\bar{1}}$ equipped with a linear map
\begin{align}
\begin{split}
V & \to\End(M)[[z^{\pm1}]] \\
a & \mapsto Y_M(a,z)=\sum_{n\in\ZZ} a_{(n),M} z^{-n-1}
\end{split}
\end{align}
which satisfies the following axioms for any $a,b\in V$, $u\in M$:
\begin{enumerate}
\item $Y_M(a,z)u\in M\ldp z\rdp$ and if $a\in V_{\bar{0}}$ (resp. $a\in V_{\bar{1}}$) then $a_{(n),M}$ is an even (resp. odd) operator for all $n$;
\item $Y_M(\mathbf{1},z)=\Id_M$; and
\item if $a\in V_{p(a)}$ and $b\in V_{p(b)}$, there exists an element
$$f\in M[[z, w]][z^{-1},w^{-1},(z-w)^{-1}]$$
depending on $a$, $b$, and $u$, such that
$$Y_M(a,z)Y_M(b,w)u, \quad (-1)^{p(a)p(b)}Y_M(b,w)Y_M(a,z)u,$$
$$\text{and}\qquad Y_M(Y(a,z-w)b,w)u$$
are the expansions of $f$ in the corresponding spaces, $M\ldp z\rdp\ldp w\rdp$, $M\ldp w\rdp\ldp z\rdp$, and $M\ldp w\rdp\ldp z-w\rdp$, respectively.
\end{enumerate}

One can also define modules which are twisted by a symmetry of the vertex algebra; we shall use the following special case. Let $\theta:=\Id_{V_{\bar{0}}}\oplus(-\Id_{V_{\bar{1}}})$ be the parity involution on a super vertex operator algebra $V=V_{\bar{0}}\oplus V_{\bar{1}}$. A \emph{canonically-twisted module} over $V$ is a super vector space $M=M_{\bar{0}}\oplus M_{\bar{1}}$ equipped with a linear map
\begin{align}
\begin{split}
V & \to\End(M)[[z^{\pm{1/2}}]] \\
a & \mapsto Y_\tw(a,z^{1/2})=\sum_{n\in\frac{1}{2}\ZZ}a_{(n),\tw}z^{-n-1},
\end{split}
\end{align}
associating to each $a\in V$ a \emph{twisted vertex operator} $Y_\tw(a,z^{1/2})$, which satisfies the following axioms for any $a,b\in V, u\in M$:
\begin{enumerate}
\item $Y_\tw(a,z^{1/2})u\in M\ldp z^{1/2}\rdp$ and if $a\in V_{\bar{0}}$ (resp. $a\in V_{\bar{1}}$) then $a_{(n),\tw}$ is an even (resp. odd) operator for all $n$;
\item $Y_\tw(\mathbf{1},z^{1/2})=\Id_M$;
\item if $a\in V_{p(a)}$ and $b\in V_{p(b)}$, there exists an element
$$f\in M[[z^{1/2}, w^{1/2}]][z^{-{1/2}},w^{-{1/2}},(z-w)^{-1}]$$
depending on $a$, $b$, and $u$, such that
$$Y_\tw(a,z^{1/2})Y_\tw(b,w^{1/2})u, \quad (-1)^{p(a)p(b)}Y_\tw(b,w^{1/2})Y_\tw(a,z^{1/2})u,$$
$$\text{and}\qquad Y_\tw(Y(a,z-w)b,w^{1/2})u$$
are the expansions of $f$ in the three spaces $M\ldp z^{1/2}\rdp\ldp w^{1/2}\rdp$, in $M\ldp w^{1/2}\rdp\ldp z^{1/2}\rdp$, and in $M\ldp w^{1/2}\rdp\ldp z-w\rdp$, respectively; and
\item if $\theta(a)=(-1)^ma$, then $a_{(n),\tw}=0$ for $n\notin\ZZ+\frac{m}{2}$.
\end{enumerate}
More details can be found in \cite{MR1372724}.

The notion of super vertex algebra may be refined by introducing representations of certain Lie algebras. The \emph{Virasoro} algebra is the Lie algebra spanned by $L(m), m\in\ZZ$ and a central element ${\bf c}$, with Lie bracket
\begin{gather}\label{eqn:LmLn}
[L(m),L(n)]=(m-n)L(m+n)+\frac{m^3-m}{12}\delta_{m+n,0}{\bf c}.
\end{gather}
A super vertex \emph{operator} algebra is a super vertex algebra containing a \emph{Virasoro element} (or \emph{conformal element}) $\omega\in V_{\bar{0}}$ such that if $L(n):=\omega_{(n+1)}$ for $n\in \ZZ$ then 
\begin{enumerate}\setcounter{enumi}{4}
\item $L({-1})=T$;
\item $[\Lm, \Ln]=(m-n)L_{m+n}+\frac{m^3-m}{12}\delta_{m+n,0}c\Id_V$ for some $c\in \CC$, called the {\em central charge} of $V$; 
\label{item:Vircond}
\item $\Lo$ is a diagonalizable operator on $V$, with eigenvalues contained in $\frac{1}{2}\ZZ$ and bounded from below, and with finite-dimensional eigenspaces; and
\item the super space structure on $V$ is recovered from the $\Lo$-eigendata according to the rule that $p(a)=2n\pmod{2}$ when $L(0)v=nv$.
\end{enumerate}
According to item \ref{item:Vircond}, the components of $Y(\omega,z)$ generate a representation of the Virasoro algebra on $V$ with central charge $c$. 

For $V$ a super vertex operator algebra we write 
\begin{gather}\label{eqn:va:funds-degdec}
	V=\bigoplus_{n\in \frac{1}{2}\ZZ}V_n,\quad 
	V_n=\{v\in V\mid L(0)v=nv\},
\end{gather}
for the decomposition of $V$ into eigenspaces for $L(0)$, and we call $V_n$ the homogeneous subspace of {\em degree} $n$.

Following \cite{MR1615132,MR2208812}, a $V$-module $M=(M,Y_M)$ for a super vertex operator algebra $V$ is called {\em admissible} if there exists a grading $M=\bigoplus_{n\in\frac{1}{2}\ZZ}M(n)$, with $M(n)=\{0\}$ for $n<0$, such that $a_{(n)}M(k)\subset M(k+m-n-1)$ when $a\in V_m$. An admissible $V$-module is {\em irreducible} if it has no non-trivial proper graded submodules. A super vertex operator algebra $V$ is called {\em rational} if any admissible module is a direct sum of irreducible admissible modules, and we say that $V$ is {\em self-dual} if $V$ is rational, irreducible as a $V$-module, and if $V$ is the only irreducible admissible $V$-module, up to isomorphism.

There are two particularly important extensions of the Virasoro algebra to a super Lie algebra. The \emph{Neveu-Schwarz} algebra is a super Lie algebra spanned by $L(m), m\in\ZZ$, $G(n+1/2),n\in\ZZ$, and a central element ${\bf c}$; the $L(m)$ and ${\bf c}$ span the even subalgebra, isomorphic to the Virasoro algebra, and the $G(n+1/2)$ span the odd subspace. The Lie bracket is defined by \eqref{eqn:LmLn} and
\begin{gather}
[L(m),G(n+1/2)]=\frac{m-2(n+1/2)}{2}G(m+n+1/2) \label{eqn:LmGn+1/2},\\
[G(m+1/2),G(n-1/2)]=2L(m+n)+\frac{4(m+1/2)^2-1}{12}\delta_{m+n,0}{\bf c} \label{eqn:Gm+1/2Gn-1/2}.
\end{gather}

An \emph{$N=1$} super vertex operator algebra is a super vertex algebra containing an \emph{$N=1$ element} $\tau\in V_{\bar{1}}$ such that if $G(n+1/2):=\tau_{(n+1)}$ for $n\in\ZZ$ then $\omega:=\frac{1}{2}G(-1/2)\tau$ is a Virasoro element (with components $L(n):=\omega_{(n+1)}$) as above, and the $L(m), G(n+1/2)$ generate a representation of the Neveu-Schwarz algebra; in particular, the $L(m), G(n+1/2)$ satisfy \eqref{eqn:LmLn}, \eqref{eqn:LmGn+1/2}, and \eqref{eqn:Gm+1/2Gn-1/2}, where the role of ${\bf c}$ is played by $c\Id_V$ for some $c\in\CC$. For further discussion we refer to \cite{MR1302018}.

Another extension of the Virasoro algebra to a super Lie algebra is the \emph{Ramond} algebra, spanned by $L(m), m\in\ZZ$, $G(n), n\in\ZZ$, and a central element ${\bf c}$; as in the case of the Neveu-Scwarz algebra the $L(m)$ and ${\bf c}$ span the even subalgebra, isomorphic to the Virasoro algebra, and the $G(n)$ span the odd subspace. The Lie bracket is defined by \eqref{eqn:LmLn} and
\begin{gather}
[L(m),G(n)]=\frac{m-2n}{2}G(m+n) \label{eqn:LmGn}, \\
[G(m),G(n)]=2L(m+n)+\frac{4m^2-1}{12}\delta_{m+n,0}{\bf c}. \label{eqn:GmGn}
\end{gather}

If $V$ is an $N=1$ super vertex operator algebra (with $N=1$ element $\tau$ and Virasoro element $\omega=\frac{1}{2}G(-1/2)\tau$) and $M$ is a canonically-twisted module for $V$, then the operators $L(m):=\omega_{(n+1),\tw}$ and $G(n) := \tau_{(n+1/2),\tw}$ generate a representation of the Ramond algebra on $M$.

\subsection{Invariant Bilinear Forms}\label{sec:va:ibfs}

The notion of an invariant bilinear form on a vertex operator algebra module was introduced in \cite{MR1142494}. We say that a bilinear form $\lab\cdot\,,\,\cdot\rab:M\otimes M\to \CC$ on a module $(M,Y_M)$ for a super vertex operator algebra $V$ is {\em invariant} if 
\begin{gather}\label{eqn:va:funds-invbilfrm}
	\left\lab Y_M(a,z)b,c\right\rab
	=
	\left\lab b,Y_M^{\dag}(a,z)c\right\rab
\end{gather}
for $a\in V$ and $b,c\in M$, where $Y_M^{\dag}(a,z)$ denotes the {\em opposite} vertex operator, defined by setting
\begin{gather}\label{eqn:va:funds-oppvops}
	Y_M^{\dag}(a,z):=(-1)^nY_M(e^{zL(1)}z^{-2L(0)}a,z^{-1})
\end{gather}
for $a$ in $V_{n-1/2}$ or $V_n$. In the right-hand side of (\ref{eqn:va:funds-oppvops}) we have extended the definition of $Y_M$ from $V$ to $V\ldp z\rdp$ by requiring $\CC\ldp z\rdp$-linearity. That is, we define $Y_M(f(z)a,z)=f(z)Y_M(a,z)$ for $f(z)\in \CC\ldp z\rdp$ and $a\in V$.

Suppose that $M=\bigoplus_{n\in\frac{1}{2}\ZZ}M(n)$ is an admissible $V$-module. Then the {\em restricted dual} of $M$ is the graded vector space $M'=\bigoplus_{n\in \frac{1}{2}\ZZ}M'(n)$ obtained by setting $M'(n)=M(n)^*:=\hom_{\CC}(M(n),\CC)$. According to Proposition 2.5 of \cite{Dun_VACo} (see also Lemma 2.7 of \cite{MR2208812}, and Theorem 5.2.1 of \cite{MR1142494}), $M'$ is naturally an admissible $V$-module, called the {\em contragredient} of $M$. To define the $V$-module structure on $M'$ write $(\cdot\,,\cdot)_M$ for the natural pairing $M'\otimes M\to\CC$, and define $Y_M':V\to \End(M')[[z^{\pm 1]}]]$---the vertex operator correspondence {\em adjoint} to $Y_M$---by requiring that
\begin{gather}
	(Y_M'(a,z)b',c)_M=(b',Y^{\dag}(a,z)c)_M
\end{gather} 
for $a\in V$, $b'\in M'$ and $c\in M$. 

Following the discussion in \S5.3 of \cite{MR1142494} we observe that the datum of a non-degenerate invariant bilinear form on an admissible $V$-module $M$ is the same as the datum of a $V$-module isomorphism $M\to M'$. For if $\phi:M\to M'$ is a $V$-module isomorphism then we obtain a bilinear form $\lab\cdot\,,\cdot\rab$ on $M$ by setting $\lab b,c\rab=(\phi(b),c)_M$ for $b,c\in M$. It is easily seen to be invariant and non-degenerate. Conversely, if $\lab\cdot\,,\cdot\rab$ is a non-degenerate invariant bilinear form on $M$ then invariance implies that $\lab M(m),M(n)\rab\subset\{0\}$ unless $m=n$ (cf. Proposition 2.12 of \cite{MR1610719}), and so we obtain a linear, grading preserving isomorphism $\phi:M\to M'$ by requiring $(\phi(b),c)_M=\lab b,c\rab$ for $b,c\in M(n)$, $n\in \frac{1}{2}\ZZ$. The invariance of $\lab\cdot\,,\cdot\rab$ then implies $\phi(Y_M(a,z)b)=Y'_M(a,z)\phi(b)$ for $a\in V$ and $b\in M$, so $\phi$ is an isomorphism of $V$-modules.

The following theorem of Scheithauer is the super vertex operator algebra version of a result first proved for vertex operator algebras by Li in \cite{MR1303287}.
\begin{theorem}[\cite{MR1610719}]\label{thm:va:funds-invbilfms}
The space of invariant bilinear forms on a super vertex operator algebra $V$ is naturally isomorphic to the dual of $V_0/L(1)V_1$.
\end{theorem}

Note that there is some flexibility available in the definitions of invariant bilinear form and opposite vertex operator in the super case. For in \cite{MR1610719}, a bilinear form $\lab\cdot\,,\,\cdot\rab^*$ is said to be invariant if 
\begin{gather}\label{eqn:va:funds-invbilfrmalt}
	\left\lab Y(a,z)b,c\right\rab^*=(-1)^{|a||b|}\left\lab b,Y^*(a,z)c\right\rab^*
\end{gather}
for $Y^*(a,z):=Y(e^{-\lambda^{-2} zL(1)}(-\lambda^{-1}z)^{-2L(0)}a,-\lambda^2 z^{-1})$. Taking $\lambda=\pm \ii$ we recover the usual notion of opposite vertex operator for a vertex algebra (cf. (5.2.4) of \cite{MR1142494}), upon restriction to the even sub vertex algebra of $V$. Observe that this notion of invariant bilinear form is equivalent to (\ref{eqn:va:funds-invbilfrm}). For if we take $\lambda=-\ii$ in the definition of $Y^*$, for example, then given a bilinear form $\lab\cdot\,,\,\cdot\rab^*$ satisfying (\ref{eqn:va:funds-invbilfrmalt}), we obtain a bilinear form $\lab\cdot\,,\,\cdot\rab$ that satisfies (\ref{eqn:va:funds-invbilfrm}), upon setting $\lab a,b\rab=\lab a,b\rab^*$ for $p(a)=0$, and $\lab a,b\rab=-\ii\lab a,b\rab^*$ for $p(a)=1$. The case that $\lambda=\ii$ is directly similar.

\subsection{Clifford Module Construction}\label{sec:va:cliff}

We now review the standard construction of vertex operator algebras via Clifford algebra modules.

Let $\a$ be a finite dimensional complex vector space equipped with a non-degenerate symmetric bilinear form $\langle\cdot\,,\cdot\rangle$. For each $n\in\ZZ$ let $\a(n+1/2)$ be a vector space isomorphic to $\a$, with a chosen isomorphism $\a\to\a(n+1/2)$, denoted $u\mapsto u(n+1/2)$, and define
\be
\aa=\bigoplus_{n\in\ZZ}\a(n+1/2).
\ee
We can extend $\langle\cdot\,,\cdot\rangle$ to a non-degenerate symmetric bilinear form on $\aa$ by $\langle u(r),v(s)\rangle=\langle u,v\rangle\delta_{r+s,0}$. We obtain a {\em polarization} of $\aa$ with respect to this bilinear form---that is, a decomposition $\aa=\aa^-\oplus \aa^+$ into a direct sum of maximal isotropic subspaces---by setting
\be
\aa^-=\bigoplus_{n< 0}\a(n+1/2)\text{ and }\aa^+=\bigoplus_{n\ge 0}\a(n+1/2).
\ee

Define the \emph{Clifford algebra} of $\aa$ by $\Cliff(\aa)=T(\aa)/I(\aa)$, where $T(\aa)$ is the tensor algebra of $\aa$, with unity denoted $\mathbf{1}$, and $I(\aa)$ is the (two-sided) ideal of $T(\aa)$ generated by elements of the form $u\otimes u+\langle u,u\rangle\mathbf{1}$ for $u\in\aa$. Denote by $B^-$ and $B^+$ the subalgebras of $\Cliff(\aa)$ generated by $\aa^-$ and $\aa^+$, respectively. The linear map $-\Id$ on $\aa$ induces an involution $\theta$ on $\Cliff(\aa)$ according to the universal property of Clifford algebras. We call $\theta$ the {\em parity} involution and write $\Cliff(\aa)=\Cliff(\aa)^0\oplus \Cliff(\aa)^1$ for the corresponding decomposition into eigenspaces, where $\Cliff(\aa)^j$ denotes the $\theta$-eigenspace with eigenvalue $(-1)^j$. 

Let $\CC\vv$ be a one-dimensional vector space equipped with the trivial action from $B^+$, i.e. $\mathbf{1}\vv=\vv$ and $u\vv=0$ for any $u\in\aa^+$. Define $A(\a)$ to be the induced $\Cliff(\aa)$-module, $A(\a)=\Cliff(\aa)\otimes_{B^+}\CC\vv$. We have a natural isomorphism of $B^-$-modules
\begin{gather}\label{eqn:va:cliff-Aisomwedge}
	A(\a)\simeq \bigwedge(\aa^-)\vv.
\end{gather}
For $a\in\a$, define a vertex operator for $a(-1/2)\vv$ by
\be
Y(a(-1/2)\vv,z)=\sum_{n\in\ZZ}a(n+1/2)z^{-n-1}.
\ee
There is a {\em reconstruction theorem} (Theorem 4.4.1 of \cite{MR2082709}) which ensures that these vertex operators extend uniquely to a super vertex algebra structure on $A(\a)$. The super space structure $A(\a)=A(\a)^0\oplus A(\a)^1$ is given by the parity decomposition on $\bigwedge(\aa^-)\vv$. That is, restricting the isomorphism of (\ref{eqn:va:cliff-Aisomwedge}) we have
\begin{gather}\label{eqn:va:cliff-Aparitywedge}
	A(\a)^0\simeq \bigwedge^{\rm even}(\aa^-)\vv,\quad
	A(\a)^1\simeq \bigwedge^{\rm odd}(\aa^-)\vv.
\end{gather}

Choose an orthonormal basis $\{e_i:1\le i\le \dim \a \}$ for $\a$. The Virasoro element
\be\label{eqn:va:cm-vir}
\omega=-\frac{1}{4}\sum_{i=1}^{\dim \a }e_i(-3/2)e_i(-1/2)\vv
\ee
gives $A(\a)$ the structure of a super vertex operator algebra with central charge $c=\frac{1}{2}\dim\a$.

Observe that $A(\a)_0$ is spanned by the vacuum ${\vv}$. We compute $L(1)a=0$ for all $a\in A(\a)_1$, and conclude from Theorem \ref{thm:va:funds-invbilfms} that there is a unique non-zero invariant bilinear form on $A(\a)$ up to scale. Scale it so that $\lab \vv,\vv\rab=1$. Then taking $a=u(-1/2)\vv$ for $u\in \a$ we compute $Y'(a,z)=-Y(a,z^{-1})z^{-1}$ (cf. (\ref{eqn:va:funds-oppvops})), and conclude from (\ref{eqn:va:funds-invbilfrm}) that 
\begin{gather}\label{eqn:va:cliff-invbilfrm}
\lab u(-m-1/2)a,b\rab+\lab a,u(m+1/2)b\rab=0
\end{gather}
for $u\in \a$, $m\in \ZZ$, and $a,b\in A(\a)$. This identity is useful for computations. For example, taking $a=\vv$ and $b=v(-m-1/2)\vv$ for $v\in\a$, we see that
\begin{gather}
\lab u(-m-1/2)\vv,v(-m-1/2)\vv\rab=\lab u,v\rab
\end{gather}
for $u,v\in\a$ and $m\geq 0$.

A construction similar to $A(\a)$ produces a canonically-twisted module for $A(\a)$, which we call $A(\a)_\tw$. For the sake of simplicity, let us assume that the dimension of $\a$ is even.

For each $n\in\ZZ$ let $\a(n)$ be a vector space isomorphic to $\a$, with a chosen isomorphism $\a\to\a(n)$ denoted $u\mapsto u(n)$, and define
\be
\aa_\tw=\bigoplus_{n\in\ZZ}\a(n).
\ee
The bilinear form on $\a$ extends to a bilinear form on $\aa_\tw$ in the same way as $\aa$; namely, $\lab u(m),v(n)\rab=\lab u,v\rab\delta_{m+n,0}$.  We again require a decomposition of $\aa_\tw$ into maximal isotropic subspaces $\aa_\tw^+\oplus\aa_\tw^-$. For this first choose a polarization $\a=\a^-\oplus \a^+$, and then define $\aa_{\tw}^{\pm}$ by setting
\be
\aa_\tw^-=\a(0)^-\oplus\left(\bigoplus_{n<0}\a(n)\right)\text{ and }\aa_\tw^+=\a(0)^+\oplus\left(\bigoplus_{n>0}\a(n)\right)
\ee
where $\a(0)^{\pm}$ is the image of $\a^{\pm}$ under the isomorphism $u\mapsto u(0)$.

Denote by $B_\tw^-$ and $B_\tw^+$ the subalgebras of $\Cliff(\aa_\tw)$ generated by $\aa_\tw^-$ and $\aa_\tw^+$, respectively. Define the trivial action of $B_\tw^+$ on a one-dimensional space $\CC\vv_\tw$, and set $A(\a)_\tw=\Cliff(\aa_\tw)\otimes_{B_\tw^+}\CC\vv_\tw$. There is a natural $B_\tw^-$-module isomorphism
\begin{gather}\label{eqn:va:AtwIsomExtAlg}
	A(\a)_\tw\simeq \bigwedge(\aa_\tw^-)\vv_\tw.
\end{gather}
For $a\in\a$, define a twisted vertex operator for $a(-1/2)\vv\in A(\a)$ on $A(\a)_\tw$ by
\be
Y_\tw(a(-1/2)\vv,z^{1/2})=\sum_{n\in\ZZ} a(n)z^{-n-1/2}.
\ee
An analogue of the reconstruction theorem for modules (cf. \cite{MR2074176}) ensures that this collection of twisted vertex operators extends uniquely to a canonically-twisted $A(\a)$-module structure on $A(\a)_\tw$. In particular, the twisted vertex operator
\be
Y_\tw(\omega,z^{1/2})=\sum_{n\in\ZZ} \Ln z^{-n-2}
\ee
equips $A(\a)_\tw$ with a representation of the Virasoro algebra, and $\Lo=\omega_{(1),\tw}$ acts diagonalizably. An explicit computation yields that the eigenvalues of $\Lo$ on $A(\a)_\tw$ are contained in $\ZZ+\frac{1}{16}\dim\a$.

The finite dimensional Clifford algebra $\Cliff(\a)$ embeds in $\Cliff(\aa_\tw)$ as the subalgebra generated by $\a(0)$. Through this identification, $\Cliff(\a)$ acts on $A(\a)_\tw$, and the $\Cliff(\a)$-submodule of $A(\a)_\tw$ generated by $\vv_\tw$ is the unique (up to isomorphism) non-trivial irreducible representation of $\Cliff(\a)$. We shall denote this subspace of $A(\a)_\tw$ by $\CM$. By restricting the isomorphism of (\ref{eqn:va:AtwIsomExtAlg}) we obtain 
\begin{gather}\label{eqn:va:AtwwedgeCM}
	A(\a)_\tw\simeq\bigwedge\left(\bigoplus_{n<0}\a(n)\right)\otimes \CM,
	\quad
	\CM\simeq \bigwedge(\a(0)^-)\vv_\tw.
\end{gather}

There is a unique (up to scale) bilinear form $\lab\cdot\,,\cdot\rab_{\tw}$ on $\CM$ satisfying $\lab ua,b\rab_\tw+\lab a,ub\rab_\tw=0$ for $u\in \a$ and $a,b\in \CM$. To choose a scaling, let $\{a^-_i\}_{i=1}^{c}$ be a basis for $\a^-$, where $c=\frac{1}{2}\dim \a $, and set 
\begin{gather}\label{eqn:va:cliff-invbilfrmtw1}
\lab a^-_1\cdots a^-_{c}\vv_\tw,\vv_\tw\rab_\tw =1.
\end{gather}
We may extend this form uniquely to a bilinear form $\lab\cdot\,,\,\cdot\rab_\tw$ on $A(\a)_\tw$ by requiring that
\begin{gather}\label{eqn:va:cliff-invbilfrmtw2}
\lab u(-m)a,b\rab_\tw+\lab a,u(m)b\rab_\tw=0
\end{gather}
for $u\in \a$, $m\in \ZZ$, and $a,b\in A(\a)_\tw$. (Cf. (\ref{eqn:va:cliff-invbilfrm}).)

\section{Groups}\label{sec:gps}

In \S\ref{sec:gps:spin} we discuss the spin group of a complex vector space of even dimension, and in \S\ref{sec:gps:conway} we recall the definition and some features of the automorphism group of the Leech lattice, also known as the Conway group.

\subsection{The Spin Groups}\label{sec:gps:spin}

Define the {\em main anti-automorphism} $\alpha$ on $\Cliff(\a)$ by setting $\alpha(u_1\cdots u_k):=u_k\cdots u_1$ for $u_i\in\a$. Recall that the {\em spin group} of $\a$, denoted $\Spin(\a)$, is the set of even, invertible elements $x\in\Cliff(\a)$ with $\alpha(x)x={\bf 1}$ (i.e. the unit element of $\Cliff(\a)$) such that $xux^{-1}\in \a$ whenever $u\in\a$.

It is useful to be able to construct some elements of $\Spin(\a)$ explicitly. The expressions $\frac{1}{2}(uv-vu)\in\Cliff(\a)$, for $u,v\in\a$, span a $\binom{\dim \a}{2}$-dimensional subspace $\g<\Cliff(\a)$ which closes under the commutator $[x,y]=xy-yx$ on $\Cliff(\a)$, and forms a simple Lie algebra of type $D_c$, for $c=\frac{1}{2}\dim\a$. (Recall our assumption that $\dim\a$ is even.) The exponentials $\exp(\frac{1}{2}(uv-vu))\in\Cliff(\a)$ generate $\Spin(\a)$. For example, if $a^+,a^-\in\a$ are chosen so that 
\begin{gather}\label{eqn:gps:spin-apmevecs}
\lab a^{\pm},a^{\pm}\rab=0,\quad \lab a^-,a^+\rab =1,
\end{gather} 
then $X=\frac{\ii}{2}(a^-a^+-a^+a^-)$ satisfies $X^2=-{\bf 1}$, so $e^{\alpha X}=(\cos \alpha){\bf 1}+(\sin \alpha)X$.

Set $x(u)=xux^{-1}$ for $x\in \Spin(\a)$ and $u\in \a$. Then $u\mapsto x(u)$ is a linear transformation on $\a$ belonging to $\SO(\a)$ and the assignment $x\mapsto x(\cdot)$ defines a map $\Spin(\a)\to \SO(\a)$ with kernel $\{\pm \mathbf{1}\}$. Say that $\wh{g}\in \Spin(\a)$ is a {\em lift} of $g\in \SO(\a)$ if $\wh{g}(\cdot)=g$. For 
\begin{gather}\label{eqn:gps:spin-X}
X=\frac{\ii}{2}(a^-a^+-a^+a^-)
\end{gather}
with $a^{\pm}$ as in (\ref{eqn:gps:spin-apmevecs}) we have $Xa^{\pm}=\pm \ii a^{\pm}=-a^{\pm}X$ in $\Cliff(\a)$, so 
\begin{gather}\label{eqn:gps:spin-eaX}
	e^{\alpha X}(a^{\pm})=
	e^{\alpha X}a^{\pm}e^{-\alpha X}=e^{\pm 2\alpha \ii}a^{\pm},
\end{gather}
which is to say, $e^{\alpha X}$ is a lift of the orthogonal transformation on $\a$ which acts as multiplication by $e^{\pm 2\alpha \ii}$ on $a^{\pm}$, and as the identity on vectors orthogonal to $a^+$ and $a^-$. For future reference we note here also that $X\vv_\tw=\ii\vv_\tw$, so the action of $e^{\alpha X}\in\Spin(\a)$ on $\vv_\tw$ is given by
\begin{gather}\label{eqn:gps:spin-eaXvvtw}
	e^{\alpha X}\vv_\tw=e^{\alpha \ii}\vv_\tw. 
\end{gather}

The group $\Spin(\a)$ acts naturally on $A(\a)$ and $A(\a)_\tw$. Indeed, writing $A(\a)_1$ for the $\Lo$-eigenspace of $A(\a)$ with eigenvalue equal to $1$, the map $u(-\frac{1}{2})v(-\frac{1}{2})\vv\mapsto \frac{1}{2}(uv-vu)$ defines an isomorphism of vector spaces $A(\a)_1\to \g$, which becomes an isomorphism of Lie algebras once we equip $A(\a)_1$ with the bracket $[X,Y]:=X_{(0)}Y$. (It follows from the vertex algebra axioms that $[a_{(0)},b_{(n)}]=(a_{(0)}b)_{(n)}$ in $\End A(\a)$, for any $a,b\in A(\a)$ and $n\in \ZZ$.) Accordingly, the exponentials $e^{X_{(0)}}$ and $e^{X_{(0),\tw}}$ for $X\in A(\a)_1$ generate an action of $\Spin(\a)$ on $A(\a)$ and $A(\a)_{\tw}$, respectively. Explicitly, if $a\in A(\a)$ has the form $a=u_1(-n_1+\tfrac12)\cdots u_k(-n_k+\tfrac12)\vv$ for some $u_i\in\a$ and $n_i\in\ZZ^+$, then 
\begin{gather}\label{eqn:gps:spin-defnxa}
xa=u_1'(-n_1+\tfrac12)\cdots u_k'(-n_k+\tfrac12)\vv,
\end{gather}
for $x\in \Spin(\a)$, where $u_i'=x(u_i)$. Evidently $-\mathbf{1}$ is in the kernel of this assignment $\Spin(\a)\to\Aut(A(\a))$, so the action factors through $\SO(\a)$.

For $A(\a)_\tw$ we use (\ref{eqn:va:AtwwedgeCM}) to identify the elements of the form 
\begin{gather}
u_1(-n_1)\cdots u_k(-n_k)\otimes y
\end{gather} 
as a spanning set, where $u_i\in\a$ and $n_i\in\ZZ^+$ as above, and $y\in \CM$. The image of such an element under $x\in \Spin(\a)$ is given by $u_1'(-n_1)\cdots u_k'(-n_k)\otimes xy$, where $u_i'=x(u_i)$ as before. Since $\CM$ is a faithful $\Spin(\a)$-module, so too is $A(\a)_\tw$. 

In terms of the vertex operator correspondences we have
\begin{gather}\label{eqn:gps:spin-spinactsvops}
\begin{split}
Y(xa,z)xb&=xY(a,z)b=\sum_{n\in\ZZ}x(a_{(n)}b)z^{-n-1},\\
Y_\tw(xa,z^{1/2})xc&=xY_\tw(a,z^{1/2})c=\sum_{n\in\tfrac12\ZZ}x(a_{(n),\tw}c)z^{-n-1},
\end{split}
\end{gather}
for $x\in \Spin(\a)$, $a,b\in A(\a)$ and $c\in A_{\tw}(\a)$.

Recall that our construction of $\CM$ depends upon a choice of polarization $\a=\a^-\oplus \a^+$. Observe that if $x\in \Spin(\a)$ is a lift of $-\Id_\a\in\SO(\a)$ then the vector $\vv_\tw\in \CM$, characterized by the condition $u\vv_\tw=0$ for all $u\in\a^+$, satisfies $x\vv_\tw=\pm \ii^c\vv_\tw$, where $c=\frac{1}{2}\dim\a$. Indeed, if $\{a^{\pm}_i\}$ is a basis for $\a^\pm$, chosen so that $\lab a_i^-,a_j^+\rab=\delta_{i,j}$, then 
\begin{gather}\label{eqn:gps:spin-constzz}
\zz:=\prod_{i=1}^c e^{\frac{\pi}{2}X_i}
\end{gather}
is a lift of $-\Id_\a$, for $X_i=\frac{\ii}{2}(a_i^-a_i^+-a_i^+a_i^-)$, according to (\ref{eqn:gps:spin-eaX}). From (\ref{eqn:gps:spin-eaXvvtw}) it follows that $\zz\vv_\tw=\ii^c\vv_\tw$.

Thus we see that a choice of polarization $\a=\a^-\oplus\a^+$ distinguishes one of the two lifts of $-\Id_\a$ to $\Spin(\a)$; namely, the unique element $\zz\in\Spin(\a)$ such that $\zz(\cdot)=-\Id_\a$ and  
\begin{gather}\label{eqn:gps:spin-zzdefn}
\zz\vv_\tw=\ii^c\vv_\tw
\end{gather}
where $c=\frac12\dim \a$. We call this $\zz$ the lift of $-\Id_\a$ {\em associated} to the polarization $\a=\a^-\oplus \a^+$. The element $\zz$ acts with order two on $A(\a)_\tw$ when $\dim\a$ is divisible by $4$. In this case we write
\begin{gather}\label{eqn:gps:spin-zzespcs}
A(\a)_\tw=A(\a)_\tw^0\oplus A(\a)_\tw^1
\end{gather}
for the decomposition into eigenspaces for $\zz$, where $\zz$ acts as $(-1)^j\Id$ on $A(\a)_\tw^j$. The element $\zz$ is central so the action of $\Spin(\a)$ on $A(\a)_\tw$ preserves the decomposition (\ref{eqn:gps:spin-zzespcs}). 

Note here the difference between writing $-(xa)$ and $(-x)a$ for $x\in \Spin(\a)$ and $a\in A(\a)$. The former is just the additive inverse of the vector $xa$ in $A(\a)$, whereas the latter is the image of $a$ under the action of $-x=(-{\bf 1})x$, an element of $\Spin(\a)<\Cliff(\a)$. So, for example, $\zz a=(-\zz)a=a$ for $a\in A(\a)$, and in particular $(-\zz)a\neq -a$ unless $a=0$. On the other hand $\zz(u)=(-\zz)(u)=-u$ for $u\in \a$. So, from the description (\ref{eqn:gps:spin-defnxa}), we see that writing $A(\a)^j$ for the $(-1)^j$ eigenspace of either $\zz$ or $-\zz$ recovers the super space decomposition 
\begin{gather}
A(\a)=A(\a)^0\oplus A(\a)^1
\end{gather}
of $A(\a)$. (Cf. (\ref{eqn:va:cliff-Aparitywedge}).)

Suppose that $\a=V\otimes_{\RR}\CC$ for some real vector space $V$, and that $\lab\cdot\,,\cdot\rab$ restricts to an $\RR$-valued bilinear form on $V\subset \a$. Then we obtain another way to determine a lift of $-\Id_\a$ to $\Spin(\a)$ by choosing an orientation $\RR^+\w\subset\bigwedge^{\dim V }(V)$ of $V$. For if $\{e_i\}_{i=1}^{\dim V }$ is an ordered basis of $V$ satisfying $\lab e_i,e_j\rab=\pm \delta_{i,j}$ then $\zz=e_1\cdots e_{\dim V }$ belongs to $\Spin(\a)$ and satisfies $\zz(\cdot)=-\Id_\a$. (Recall that $\dim\a=\dim_{\CC}\a$ is assumed to be even.) On the other hand $e_1\wedge \cdots \wedge e_{\dim V }$ either belongs to $\RR^+\w$ or $\RR^-\w$, and so we can say that $\zz$ is {\em consistent} with the chosen orientation of $V$ in the former case, and {\em inconsistent} in the latter. Since a polarization $\a=\a^-\oplus \a^+$ of $\a=V\otimes_{\RR}\CC$ also determines a lift $\zz$ of $-\Id_\a$, characterized by the condition $\zz\vv_\tw=\ii^c\vv_\tw$, we can say that it too is either consistent or not with a given orientation of $V$, according as the associated lift $\zz$ is or is not consistent.

\subsection{The Conway Group}\label{sec:gps:conway}

The Leech lattice, denoted $\LL$, is the unique self-dual positive-definite even lattice of rank $24$ with no roots. That is, $\lab\lambda,\lambda\rab<4$ for $\lambda\in\LL$ implies $\lambda=0$. It was discovered in 1965 by Leech \cite{Lee_SphPkgHgrSpc,Lee_SphPkgs}, and the uniqueness statement is a consequence of a (somewhat stronger) theorem due to Conway \cite{Con_ChrLeeLat}. 

Conway also calculated \cite{MR0237634,MR0248216} the automorphism group of $\LL$, which turns out to be a non-trivial $2$-fold cover of the sporadic simple group that we denote $\Co_1$. We call $\Co_0:=\Aut(\LL)$ the {\em Conway group}. The center of $\Co_0$ is $Z(\Co_0)=\{\pm\Id\}$, and we have $\Co_1=\Co_0/Z(\Co_0)$.

Set $\LL_n:=\{\lambda\in\LL\mid \lab\lambda,\lambda\rab=2n\}$, the set of vectors of {\em type $n$} in $\LL$. Conway's uniqueness proof shows that any type $4$ vector is equivalent modulo $2\LL$ to exactly $47$ other vectors of type $4$ in $\LL$, and if $\lambda,\mu\in \LL_4$ are equivalent modulo $2\LL$ then $\lambda=\pm \mu$ or $\lab\lambda,\mu\rab=0$. Call a set $\{\lambda_i\}_{i=1}^{24}\subset\LL_4$ a {\em coordinate frame} for $\LL$ when the $\lambda_i$ are mutually orthogonal, but equivalent modulo $2\LL$. 

Set $\Omega=\{1,\ldots,24\}$ and write $\mathcal{P}(\Omega)$ for the power set of $\Omega$. Given a coordinate frame $S=\{\lambda_i\}_{i\in \Omega}$ for $\LL$, let $E=E_S$ be the subgroup of $\Co_0$ whose elements act as sign changes on the $\lambda_i$.
\begin{gather}
	E=E_S:=\left\{g\in\Co_0\mid g(\lambda_i)\in\{\pm \lambda_i\},\,\forall i\in \Omega\right\}
\end{gather}
Then $E$ is an elementary abelian $2$-group of order $2^{12}$. If we attach a subset $C(g)\subset\Omega$ to each $g\in E$ by setting 
\begin{gather}
C(g):=\{i\in \Omega\mid g(\lambda_i)=-\lambda_i\}
\end{gather}
then the symmetric difference operation equips $\GG:=\{C(g)\mid g\in E\}\subset\mathcal{P}(\Omega)$ with a group structure naturally isomorphic to that of $E$, in the sense that $g\mapsto C(g)$ is an isomorphism, $C(gh)=C(g)+C(h)$ for $g,h\in E$. The weight function $C\mapsto \#C$ equips $\GG$ with the structure of a binary linear code, and it turns out that $\GG$ is a copy of the extended binary {\em Golay code}, being the unique (cf. \cite{MR1662447,MR1667939}) self-dual doubly-even binary linear code of length $24$ with no codewords of weight $4$.

A choice of identification $\a=\LL\otimes_{\ZZ}\CC$ allows us to embed the Conway group $\Co_0=\Aut(\LL)$ in $\SO(\a)$. Given such a choice let us write $\coa$ for the corresponding subgroup of $\SO(\a)$, isomorphic to $\Co_0$. Write $g\mapsto \chi_g$ for the character of the corresponding representation of $G$.
\begin{gather}\label{eqn:gps:conway-chig}
	\chi_g:=
	\tr_{\a}g
\end{gather}

Given a subgroup $H<\SO(\a)$, say that $\wh{H}<\Spin(\a)$ is a {\em lift} of $H$ if the natural map $\Spin(\a)\to \SO(\a)$ restricts to an isomorphism $\wh{H}\xrightarrow{\sim} H$. 
\begin{proposition}\label{prop:gps:conway-lift}
Let $G<\SO(\a)$ and suppose that $G$ is isomorphic to $\Co_0$. Then there is a unique lift of $G$ to $\Spin(\a)$. 
\end{proposition}
\begin{proof}
Since the Schur multiplier of $\Co_0$ is trivial (cf. \cite{atlas}), the preimage of $\coa$ under the natural map $\Spin(\a)\to \SO(\a)$ contains a copy of $\Co_0$. So there is at least one lift. If there are two, $\coh$ and $\coh'$ say, then given $g\in \coa$, write $\wh{g}$ for the corresponding element of $\coh$, and interpret $\wh{g}'$ similarly, so that $\wh{g}(\cdot)=\wh{g}'(\cdot)=g$. Now $\wh{g}'=\pm \wh{g}$ as elements of $\Spin(\a)$, so $\coh\cap\coh'$ is a normal subgroup of $\coh$ (and of $\coh'$) containing all of its elements of odd order. The only proper non-trivial normal subgroup of $\Co_0$ its center, which has order two, so $\coh\cap\coh'=\coh$. That is, $\coh=\coh'$, as we required.
\end{proof}
Given $\coa<\SO(\a)$, isomorphic to $\Co_0$, write $\coh$ for the unique lift of $\coa$ to $\Spin(\a)$ whose existence, and uniqueness, is guaranteed by Proposition \ref{prop:gps:conway-lift}. Then $\coh$ is a copy of the Conway group acting naturally on $A(\a)$ and $A(\a)_\tw$. We write 
\begin{gather}
	\begin{split}	
	\coa&\xrightarrow{\sim}\coh\\
	g&\mapsto\wh{g}
	\end{split}
\end{gather}
for the inverse of the isomorphism $\coh\xrightarrow{\sim}\coa$ obtained by restricting the natural map $\Spin(\a)\to\SO(\a)$.

Observe that the action of $\coh\simeq \Co_0$ on $A(\a)_\tw$ depends upon the choice of polarization $\a=\a^-\oplus\a^+$ used to define $A(\a)_\tw$, for the central element of $\coh$ will be $\zz$ or $-\zz$, depending upon this choice, where $\zz$ denotes the lift of $-\Id_\a\in\SO(\a)$ to $\Spin(\a)$ associated to the chosen polarization (cf. \S\ref{sec:gps:spin}). We may assume that $\zz\in \coh$, so long as we allow ourselves to modify the polarization slightly, replacing $a^{\pm}_j$ with $a^{\mp}_j$ for some $j$, for example, given basis vectors $a^{\pm}_i\in\a^{\pm}$ satisfying $\lab a^{-}_i,a^+_j\rab=\delta_{i,j}$. (Cf. (\ref{eqn:gps:spin-constzz}).)

In practice we will take $\a^{\pm}$ to be the span of isotropic eigenvectors $a_{i}^{\pm}$ for the action of some $g\in G$ on $\a$, satisfying $\lab a_i^-,a_j^+\rab=\delta_{i,j}$. Since these conditions still hold after swapping $a_j^-$ with $a_j^+$ for some $j$, we may apply the following convention with no loss of generality: given a choice of identification $\a=\LL\otimes_{\ZZ}\CC$, with $\coa$ the corresponding copy of $\Co_0$ in $\SO(\a)$, and $\coh$ the unique lift of $\coa$ to $\Spin(\a)$, we assume that any polarization $\a=\a^-\oplus \a^+$ is chosen so that the associated lift $\zz$ of $-\Id_\a$ belongs to $\coh$.

\section{Moonshine}\label{sec:moon}

This section contains the main results of the paper. In \S\ref{sec:moon:cnstn} we describe the construction of a distinguished super vertex operator algebra $\vsn$, and its unique canonically-twisted module $\vsnt$. We equip both $\vsn$ and $\vsnt$ with actions by the Conway group $\Co_0$. We establish a characterization of the super vertex operator algebra structure on $\vsn$ in \S\ref{sec:moon:uniq}. In \S\ref{sec:moon:mtseries} we compute the graded traces attached to elements of $\Co_0$ via its actions on $\vsn$ and $\vsnt$. We identify these functions as normalized principal moduli in the case of $\vsn$, and as constant or principal moduli in the case of $\vsnt$, according as there are fixed points or not in the corresponding action on the Leech lattice.

\subsection{Construction}\label{sec:moon:cnstn}

From now on we take $\a$ to be $24$-dimensional. Given a polarization $\a=\a^-\oplus\a^+$, we let $\zz$ be the associated lift of $-\Id_\a$, so that $\zz\vv_\tw=\vv_\tw$ (cf. (\ref{eqn:gps:spin-zzdefn})). We write $A(\a)_\tw=A(\a)_\tw^0\oplus A(\a)_\tw^1$ for the decomposition of $A(\a)_\tw$ into eigenspaces for $\zz$ (cf. (\ref{eqn:gps:spin-zzespcs})). 

Because it is the even part of a super vertex algebra, $A(\a)^0$ is itself a vertex algebra, and $A(\a)_\tw$ is an (untwisted) $A(\a)^0$-module. The decomposition $A(\a)_\tw=A(\a)_\tw^0\oplus A(\a)_\tw^1$ is a decomposition into submodules for $A(\a)^0$. 

Consider the $A(\a)^0$-modules $\vsn$ and $\vsnt$ defined by setting
\begin{gather}\label{eqn:moon:cnstn-vuvv}
	\vsn=A(\a)^0\oplus A(\a)_{\tw}^1,\quad
	\vsnt=A(\a)^1\oplus A(\a)_{\tw}^0.
\end{gather}

\begin{proposition}\label{prop:moon:cnstn-vatwmodstruc}
The $A(\a)^0$-module structure on $\vsn$ extends uniquely to a super vertex operator 
algebra structure on $\vsn$, and the $A(\a)^0$-module structure on $\vsnt$ extends uniquely to a canonically-twisted $\vsn$-module structure.
\end{proposition}
\begin{proof}
We proceed along similar lines to the proof of Proposition 4.1 in \cite{Dun_VACo}. Given a choice of polarization $\a=\a^-\oplus \a^+$, the boson-fermion correspondence \cite{MR643037} defines an isomorphism of vertex operator algebras $A(\a)^0\xrightarrow{\sim} V_L$, according to \cite{MR1284796}, where $V_L$ is the lattice vertex operator algebra attached to 
\begin{gather}
L=\left\{(n_i)\in \ZZ^{12}\mid \sum n_i=0\pmod{2}\right\},
\end{gather}
being a copy of the root lattice of type $D_{12}$. It also extends to compatible isomorphisms of vertex operator algebra modules
\begin{gather}\label{eqn:moon:cnstn-bfcorr}
	A(\a)^1\xrightarrow{\sim}V_{L+\lambda_v},\;
	A(\a)^0_\tw\xrightarrow{\sim}V_{L+\lambda_s},\;
	A(\a)^1_\tw\xrightarrow{\sim}V_{L+\lambda_c},
\end{gather}
where the $\lambda_x$ are representatives for the non-trivial cosets of $L$ in its dual, $L^*=\frac{1}{2}\ZZ^{12}$,
\begin{gather}\label{eqn:moon:cnstn-lambdax}
	\lambda_v=(1,0,\cdots,0),\;
	\lambda_s=\frac{1}{2}(1,1,\cdots,1),\;
	\lambda_c=\frac{1}{2}(-1,1,\cdots,1).
\end{gather}

The irreducible modules for a lattice vertex operator algebra are known \cite{MR1245855} to be in correspondence with the cosets of the lattice in its dual---i.e. the discriminant group of the lattice, $L^*/L$---and the associated fusion algebra is naturally isomorphic to the group algebra $\CC[ L^*/L]$. (These facts are explained in detail in \cite{MR1233387}.) 

From this we deduce that $\vsn$ is isomorphic to $V_{L^+}$, as a $V_L$-module, where $L^+=L\cup (L+\lambda_c)$. Note that $L^+$ is a self-dual integral lattice in the particular case at hand. (In fact, $L^+$ is the unique self-dual positive-definite integral lattice of rank $12$ with no vectors of length $1$, sometimes denoted $D_{12}^+$.) So the super vertex operator algebra structure on $V_{L^+}$, which uniquely extends the $V_L$-module structure according to the structure of the fusion algebra of $V_L$, furnishes the claimed super vertex operator algebra structure on $\vsn$, and $\vsn$ is self-dual.

By inspection we see that the canonical automorphism of $V_{L^+}$---arising from the super structure---coincides with that attached to the vector $\lambda_v\in\frac{1}{2}V_{L^+}$, since $e^{2 \pi i \lab \lambda_v,\lambda\rab}$ is $1$ or $-1$ according as $\lambda\in L^+$ belongs to $L$ or $L+\lambda_c$. This shows that the coset module $V_{L^++\lambda_v}=V_{L+\lambda_v}\oplus V_{L+\lambda_s}$ is the unique irreducible canonically-twisted module for $V_{L^+}$, so $\vsnt$ is the unique irreducible canonically-twisted module for $\vsn$, according to (\ref{eqn:moon:cnstn-bfcorr}). This completes the proof.
\end{proof}

We now equip $\vsn$ and $\vsnt$ with module structures for the Conway group, $\Co_0$. 

As detailed in \S\ref{sec:gps:spin}, the spin group $\Spin(\a)$ acts naturally on $A(\a)^j$ and $A(\a)_\tw^j$, so it acts naturally on $\vsn$ and $\vsnt$. This action respects the super vertex algebra and canonically-twisted module structures defined in Proposition \ref{prop:moon:cnstn-vatwmodstruc}, in the sense that (\ref{eqn:gps:spin-spinactsvops}) holds for $x\in\Spin(\a)$, $a,b\in \vsn$ and $c\in \vsnt$. 

Recalling the setup of \S\ref{sec:gps:conway} we now assume to be chosen an identification $\a=\LL\otimes_{\ZZ}\CC$, and write $\coa$ for the corresponding copy of $\Co_0=\Aut(\LL)$ in $\SO(\a)$, isomorphic to $\Co_0$. We take $\coh$ be the lift of $\coa$ to $\Spin(\a)$ (cf. Proposition \ref{prop:gps:conway-lift}), so that  
the restriction of the natural map $\Spin(\a)\to\SO(\a)$ defines an isomorphism $\coh\xrightarrow{\sim}\coa$. We write $g\mapsto\wh{g}$ for the inverse isomorphism, and in this way we obtain actions of the Conway group $\coh\simeq \Co_0$ on $\vsn$ and $\vsnt$. 

Note that the actions on $\vsn$ and $\vsnt$ depend upon the choice of polarization $\a=\a^-\oplus \a^+$. We assume, according to the convention established in \S\ref{sec:gps:conway}, that the polarization is chosen so that the associated lift $\zz\in\Spin(\a)$ of $-\Id_\a\in \SO(\a)$ is the non-trivial central element of $\coh$. 
\begin{gather}\label{eqn:moon:cnstn-zzinz}
	\zz\in Z(\coh)
\end{gather}
With this convention both $\vsn$ and $\vsnt$ are faithful modules for $\coh$. 

If not $\zz\in Z(\coh)$ then we would have $-\zz\in Z(\coh)$, and $-\zz$ acts trivially on both $A(\a)^0$ and $A(\a)_\tw^1$ (cf. \S\ref{sec:gps:spin}). Thus for $-\zz\in Z(\coh)$ the $\coh$-module structure on $\vsn$ would factor through to $\coh/Z(\coh)$, being a copy of the simple group $\Co_1$ (cf. \S\ref{sec:gps:conway}). 

Actually, such an action, not faithful for $\coh$, is useful for us, and arises naturally when we consider the spaces $\vfn$ and $\vfnt$, closely related to $\vsn$ and $\vsnt$, defined by setting
\begin{gather}\label{eqn:moon:cnstn-vupvvp}
	\vfn=A(\a)^0\oplus A(\a)_{\tw}^0,\quad
	\vfnt=A(\a)^1\oplus A(\a)_{\tw}^1.
\end{gather}
Making obvious changes to the proof of Proposition \ref{prop:moon:cnstn-vatwmodstruc}, (i.e. swapping $A(\a)^1_\tw$ with $A(\a)^0_\tw$, and $\lambda_c$ with $\lambda_s$, \&c.,) we obtain the following direct analogue of that result. 
\begin{proposition}\label{prop:moon:cnstn-vatwmodstrucp}
The $A(\a)^0$-module structure on $\vfn$ extends uniquely to a super vertex operator 
algebra structure on $\vfn$, and the $A(\a)^0$-module structure on $\vfnt$ extends uniquely to a canonically-twisted $\vfn$-module structure.
\end{proposition}
Actually, $\vfn$ is isomorphic to $\vsn$ as a super vertex operator algebra, since the proof of Proposition \ref{prop:moon:cnstn-vatwmodstruc} shows that both are isomorphic to the lattice super vertex operator algebra attached to $D_{12}^+$, being the unique self-dual positive-definite integral lattice of rank $12$ with no vectors of length $1$.

The group $\coh$ acts naturally on $\vfn$ and $\vfnt$ via the natural actions of $\Spin(\a)$, and it is as $\coh$-modules that the difference between $\vsn$ and $\vfn$ manifests: the center of $\coh$ acts trivially on the latter, according to our convention (\ref{eqn:moon:cnstn-zzinz}).

To aid the reader in comparing the results here with those of \cite{Dun_VACo} we mention that the $N=1$ super vertex operator algebra $_{\CC}A^{f\natural}$ studied there is isomorphic to $\vsn\simeq \vfn$ as a super vertex operator algebra. The main results of \cite{Dun_VACo} include the statement that the full automorphism group (fixing the $N=1$ element) of $_{\CC}A^{f\natural}$ is a copy of the sporadic simple Conway group $\Co_1$. So as $\Co_0$-modules we have $_{\CC}A^{f\natural}\simeq\vfn$, but $_{\CC}A^{f\natural}\not\simeq \vsn$. A faithful $\Co_0$-module structure on the super vertex operator algebra underlying $_{\CC}A^{f\natural}$ is mentioned in Remark 4.12 of \cite{Dun_VACo}.

Following the method of \cite{MR1142494} we may describe the vertex operator algebra structure on $\vsn\simeq \vfn$ quite explicitly. For the sake of later applications we now present details for the realization $\vfn$. 

So, we seek to describe the vertex operator correspondence on $\vfn$ explicitly (cf. (\ref{eqn:va:funds-vopcorr})). According to the vertex algebra axioms we may regard this correspondence as a linear map $\vfn\otimes\vfn\to\vfn\ldp z\rdp$, denoted $a\otimes b\mapsto Y(a,z)b$, whose restriction to $A(\a)^0\otimes \vfn$ is already defined by the $A(\a)^0$-module structure on $\vfn=A(\a)^0\oplus A(\a)^0_\tw$. According to the fusion rules described in the proof of Proposition \ref{prop:moon:cnstn-vatwmodstruc}, we require to specify linear maps 
\begin{gather}
	A(\a)^0_\tw\otimes A(\a)^0\to A(\a)^0_\tw\ldp z\rdp,\label{eqn:moon:cnstn-vopsvuptu}\\
	A(\a)^0_\tw\otimes A(\a)^0_\tw\to A(\a)^0\ldp z\rdp.\label{eqn:moon:cnstn-vopsvuptt}	
\end{gather}
For (\ref{eqn:moon:cnstn-vopsvuptu}) we apply the fact that a vertex operator correspondence should satisfy skew symmetry (cf. \cite{MR1142494}) to conclude that 
\begin{gather}\label{eqn:moon:cnstn-defnvopsvuptu}
	Y(a,z)b=e^{zL(-1)}Y(b,-z)a
\end{gather}
for $a\in A(\a)^0_\tw$ and $b\in A(\a)^0$. The right-hand side of (\ref{eqn:moon:cnstn-defnvopsvuptu}) is already defined, by the $A(\a)^0$-module structure on $\vfn$, so we may regard it as defining the left-hand side. 

For (\ref{eqn:moon:cnstn-vopsvuptt}) we use the non-degenerate bilinear forms on $A(\a)$ and $A(\a)_\tw$ (cf. (\ref{eqn:va:cliff-invbilfrm}) and (\ref{eqn:va:cliff-invbilfrmtw2})), to define $Y(a,z)b$ for $a,b\in A(\a)^0_\tw$ by requiring that
\begin{gather}\label{eqn:moon:cnstn-defnvopsvuptt}
	\lab Y(a,z)b, c\rab
	=(-1)^n\lab e^{z^{-1}L(1)}b,Y(c,-z^{-1})e^{zL(1)}
	a\rab
	_\tw z^{1-2n}
\end{gather}
for all $c\in A(\a)^0$ when $a\in (A(\a)^0_\tw)_{n-1/2}$. (Cf. (\ref{eqn:va:funds-oppvops}).)

The identity (\ref{eqn:moon:cnstn-defnvopsvuptt}) ensures that the bilinear form on $\vfn$, obtained by restricting those on $A(\a)$ and $A(\a)_\tw$, is invariant (cf. (\ref{eqn:va:funds-invbilfrm})) for the given super vertex operator algebra structure. According to Theorem \ref{thm:va:funds-invbilfms}, an invariant bilinear form on $\vfn$ is uniquely determined, up to scale. Thus we arrive at the following result.
\begin{proposition}\label{prop:moon:cnstn-invbilfrm}
The super vertex operator algebra $\vfn$ admits a unique invariant bilinear form such that $\lab\vv,\vv\rab=1$. It coincides with the bilinear form obtained by restriction from those defined above for $A(\a)$ and $A(\a)_\tw$.
\end{proposition}

In \cite{Dun_VACo} it is shown that the super vertex operator algebra $\vsn\simeq \vfn$ admits an $N=1$ structure, with automorphism group isomorphic to $\Co_1$. Our next result is a kind of converse to that, showing how to recover a $\Co_1$-invariant $N=1$ structure from an action by automorphisms of $\Co_1$ on $\vfn$.

\begin{proposition}\label{prop:moon:cnstn-cotons}
Suppose to be given a non-trivial map $\Co_1\to \Aut(\vfn)$. Then the resulting action of the simple Conway group $\Co_1$ on $\vfn$ fixes a unique one-dimensional subspace of $(\vfn)_{3/2}$. A suitably scaled vector in this subspace defines an $N=1$ structure on $\vfn$.
\end{proposition}
\begin{proof}
The content of Proposition 4.6 in \cite{Dun_VACo} is that the full group of super vertex operator algebra automorphisms of $\vfn$ is $\Spin(\a)/\lab \zz\rab$. So a non-trivial action of $\Co_1$ on $\vfn$ by automorphisms realizes $\Co_1$ as a subgroup of $\Spin(\a)/\lab\zz\rab$. Write $\coh$ for the preimage of this copy of $\Co_1$ in $\Spin(\a)$. Then either $\coh\simeq 2\times \Co_1$, or $\coh$ is isomorphic to the Conway group $\Co_0$, since these are the only $2$-fold covers of $\Co_1$. In either case $\lab\zz\rab$ is the only normal subgroup of order $2$ in $\coh$, so $\coh$ has trivial intersection with the kernel of the natural map $\Spin(\a)\to \SO(\a)$. We conclude that $\coh\simeq 2\times\Co_1$ is impossible, for otherwise the map $\Spin(\a)\to \SO(\a)$ would furnish a non-trivial representation of $\Co_1$ on $\a$, and the minimal dimension of a non-trivial representation for $\Co_1$ is $276$ according to \cite{atlas}. So $\coh\simeq \Co_0$.

Write $\coa$ for the image of $\coh$ in $\SO(\a)$. Then $\coa$ is a copy of $\Co_0$ in $\SO(\a)$ and $\coh$ is a lift of $\coa$ to $\Spin(\a)$ such that $\zz\in Z(\coh)$. The group $\coa$ must preserve a copy $\LL$ of the Leech lattice in $\a$, so we are in the setup of \S\ref{sec:gps:conway}, and our notation $\coh$, $\coa$, \&c., is consistent with the conventions established there.

Now let $S=\{\lambda_i\}_{i\in\Omega}$ be a coordinate frame for $\LL\subset\a$ (cf. \S\ref{sec:gps:conway}), and let  $E=E_S$ be the subgroup of $\coa$ consisting of elements which act by sign changes on the $\lambda_i$. Let $\GG$ be the corresponding copy of the Golay code in $\mathcal{P}(\Omega)$. We will use $E$ to construct the desired $N=1$ element in $(A(\a)_\tw)_{3/2}^0$. 

As explained in \S\ref{sec:va:cliff} we may use the isomorphism $\a\xrightarrow{\sim}\a(0)$ to identify $\Cliff(\a)$ as a subalgebra of $\Cliff(\hat{\a}_\tw)$. In this way we may regard $A(\a)_\tw$ as a $\Cliff(\a)$-module, and we may identify $(A(\a)_\tw)_{3/2}=\CM$ as the $\Cliff(\a)$-submodule of $A(\a)_\tw$ generated by $\vv_\tw$.

Define an idempotent element $t\in \Cliff(\a)$ by setting 
\begin{gather}
t=\frac{1}{4096}\sum_{g\in E}\wh{g},
\end{gather} 
where $g\mapsto\wh{g}$ denotes the inverse of the natural isomorphism $\coh\to\coa$. Then $t$ is not in the subalgebra of $\Cliff(\a)<\Cliff(\hat\a_\tw)$ generated by $\a^+$, since the only idempotent in $B^+$ (cf. \S\ref{sec:va:cliff}) is ${\bf 1}$. So $t\vv_\tw\neq 0$, and 
\begin{gather}
(A(\a)_\tw)_{3/2}=\CM=\Cliff(\a)t\vv_\tw=\{xt\vv_\tw\mid x\in \Cliff(\a)\}
\end{gather}
since $\CM$ is an irreducible $\Cliff(\a)$-module.

Now choose an ordering on the index set $\Omega$, let $e_i=\frac{1}{\sqrt{8}}\lambda_i$ for $i\in\Omega$, and given a subset $C=\{i_1,\cdots,i_k\}\subset\Omega$, with $i_1<\cdots<i_k$, define an element $e_C\in \Cliff(\a)$ by setting $e_C=e_{i_1}e_{i_2}\cdots e_{i_k}$. Then, taking $e_{\emptyset}={\bf 1}$, the set $\{e_C\mid C\subset \Omega\}$ furnishes a vector space basis for $\Cliff(\aa)$, so $\CM$ is spanned by the vectors $e_Ct\vv_\tw$ for $C\subset\Omega$. Also, $e_Ce_D=\pm e_{C+D}$ (where the $+$ in the subscript on the right hand side denotes the symmetric difference operation on $\mathcal{P}(\Omega)$). 

Observe that $e_Ct=\pm e_Dt$ whenever $C$ and $D$ are equivalent modulo $\GG$, since in that case one of $e_{C+D}$ or $-e_{C+D}$ belongs to $\wh{E}=\{\wh{g}\mid g\in E\}$. So if $\mathcal{T}$ is a set of representatives for the cosets of $\GG$ in $\mathcal{P}(\Omega)$ then $\CM$ is spanned by the $e_Ct\vv_\tw$ for $C\in \mathcal{T}$.
Since the Golay code is self-dual, $\mathcal{T}$ has cardinality $2^{12}$, which is also the dimension of $\CM$, so the $e_Ct\vv_\tw$ for $C\in \mathcal{T}$ must in fact furnish a basis for $\CM$.

We claim that $t\vv_\tw$ is $\coh$-invariant. Certainly it is $\wh{E}$-invariant. Using the fact that the $e_Ct\vv_\tw$ for $C\in \mathcal{T}$ form a basis for $\CM$ we see that $t\vv_\tw$ is actually the only $\wh{E}$-invariant vector in $\CM$, because the space spanned by $e_Ct\vv_\tw$ is a one-dimensional representation of $\wh{E}$ with character given by $\chi(\wh{g})=(-1)^{\# (C\cap D)}$ in case $g=g(D)$. Since the Golay code is self-dual, we only have $\# (C\cap D)=0\pmod{2}$ for all $D\in \GG$ when $C+\GG=\GG$. 

Consider the action of $\coh$ on $\CM^0$. The central element of $\coh$ is $\zz$, which acts trivially on $\CM^0$, so $\CM^0$ is a direct sum of irreducible modules for $\coh/\lab \zz\rab\simeq \Co_1$. We have $\dim \CM^0=2048$ so only irreducible representations of $\Co_1$ with dimension not exceeding $2048$ can arise. According to \cite{atlas} there are exactly four irreducible representations, up to equivalence, that can appear, and they are each determined by their dimension: $1$, $276$, $299$ or $1771$. The equation $2048=276a+299b+1771c$ has no non-negative integer solutions, so there must be at least one non-zero $\coh$-fixed vector in $\CM^0$. Such a vector must also be fixed by $\wh{E}$, and we have seen that $t\vv_\tw$ is the only possibility, so $t\vv_\tw$ is fixed by $\coh$, as was claimed. 

It follows that the decomposition of $\CM^0$ into irreducible representations for $\Co_1$ is given by $2048=1+276+1771$, with the one dimensional representation spanned by $t\vv_\tw$. From this we may conclude that $t\vv_\tw$ is not isotropic with respect to the invariant bilinear form $\lab\cdot\,,\,\cdot\rab$ on $\vfn$. (Cf. Proposition \ref{prop:moon:cnstn-invbilfrm}.) For the restriction of $\lab\cdot\,,\,\cdot\rab$ to $\CM^0=(A(\a)_{\tw})_{3/2}$ is $\Spin(\a)$-invariant, and hence also $\coh$-invariant, and so the above decomposition implies that a $\coh$-invariant map $f:\CM^0\to \CC$ that vanishes on $t\vv_\tw$ vanishes everywhere. Take $f(v)=\lab v,t\vv_\tw\rab$ for $v\in \CM^0$ to conclude that if $t\vv_\tw$ is isotropic then $\lab v,t\vv_\tw\rab=0$ for all $v\in\CM^0$, but this contradicts the non-degeneracy of $\lab\cdot\,,\,\cdot\rab_\tw$ on $A(\a)_\tw$, which can be easily checked from the defining identities, (\ref{eqn:va:cliff-invbilfrmtw1}) and (\ref{eqn:va:cliff-invbilfrmtw2}). 

Now choose $\alpha\in \CC$ such that $\tau=\alpha t\vv_\tw$ satisfies $\lab\tau,\tau\rab=8$. Observe that $\lab e_C\tau,\tau\rab=0$ whenever $C\subset \{1,\ldots,24\}$ has cardinality two or four. We conclude from Proposition 4.3 of \cite{Dun_VACo} that $\tau$ is an $N=1$ element for $\vfn$. This completes the proof.
\end{proof}

\subsection{Characterization}\label{sec:moon:uniq}

In this short section we establish a characterization of the super vertex operator algebra structure on $\vsn$. This is a strengthening of the main theorem of \S5.1 in \cite{Dun_VACo}, for we arrive at the same conclusion without the hypothesis of an $N=1$ structure.

Recall that a super vertex operator algebra $V$ is said to be {\em $C_2$-cofinite} if the subspace $\{a_{(-2)}b\mid a,b\in V\}<V$ has finite codimension in $V$. Following \cite{MR1650625} we say that a super vertex operator algebra $V=(V,Y,\vv,\omega)$ is of {\em CFT type} if the $L(0)$-grading $V=\bigoplus_{n\in\frac{1}{2}\ZZ}V_n$ is bounded from below by $0$, and if $V_0$ is spanned by the vacuum vector $\vv$. Note that a super vertex operator algebra (in the sense of \S\ref{sec:va:funds}) that is $C_2$-cofinite and of CFT type, is {\em nice} in the sense of \cite{Hoehn2007}.

\begin{theorem}\label{thm:moon:uniq-uniq}
Let $V$ be a self-dual $C_2$-cofinite rational super vertex operator algebra of CFT type with central charge $12$ such that $L(0)v=\frac{1}{2}v$ for $v\in V$ implies $v=0$. Then $V$ is isomorphic to $\vsn$ as a super vertex operator algebra.
\end{theorem}
\begin{proof}
We first show that $V$ admits a unique invariant bilinear form. Since $V$ is self-dual, the contragredient module $V'$ is isomorphic to $V$, so there exists a $V$-module isomorphism $\phi:V\to V'$. As explained in \S\ref{sec:va:ibfs}, this determines a non-degenerate invariant bilinear form $\lab\cdot\,,\cdot\rab$ on $V$. Now let $a\in V_1$. We claim that $L(1)a=0$. In any case, $L(1)a\in V_0$, so $L(1)a=C\vv$ for some $C\in\CC$, since $V$ is of CFT type. Thus we have $Y^{\dag}(a,z)=(-1)Y(a,z^{-1})z^{-2}+(-C)\Id_V z^{-1}$. (Cf. (\ref{eqn:va:funds-oppvops}).) Applying axiom \ref{itm:va:funds-vaaxid} from the super vertex algebra definition in \S\ref{sec:va:funds} we see that $\lab Y(a,z)\vv,\vv\rab=0$, since $a_{(-n-1)}\vv\in V_{n+1}$, and $\lab V_m,V_n\rab=0$ unless $m=n$ (cf. \S\ref{sec:va:ibfs}). On the other hand $\lab Y(a,z)\vv,\vv\rab=\lab\vv,Y^{\dag}(a,z)\vv\rab$ by invariance, so 
\begin{gather}
	0=\lab\vv,Y^{\dag}(a,z)\vv\rab=
	(-1)\lab\vv,Y(a,z^{-1})\vv\rab z^{-2}+(-C)\lab\vv,\vv\rab z^{-1}.
\end{gather}
Now the coefficient of $z^{-1}$ in $Y(a,z^{-1})\vv z^{-2}$ is $a_{(0)}\vv$, which vanishes by another application of axiom \ref{itm:va:funds-vaaxid}. Since $\lab\cdot\,,\cdot\rab$ is non-degenerate, and $V_0$ is spanned by $\vv$, we must have $\lab\vv,\vv\rab\neq 0$. So we must have $C=0$. This verifies our claim that $L(1)V_1\subset\{0\}$. Now we apply Theorem \ref{thm:va:funds-invbilfms} to conclude that the invariant bilinear form $\lab\cdot\,,\cdot\rab$ just constructed is in fact the unique invariant bilinear form on $V$. 

Since $V_1$ is in the kernel of $L(1)$ we may conclude that the even sub vertex operator algebra $V_{\bar{0}}$ is {\em strongly rational} in the sense of \S3 of \cite{MR2097833}. This allows us to apply Theorem 1 of \cite{MR2097833}, which tells us that the Lie algebra structure on $V_1$, obtained by setting $[a,b]:=a_{(0)}b$ for $a,b\in V_1$, is reductive. The argument used to prove Theorem 2 of \cite{MR2097833}---see the proof of Theorem 5.12 in \cite{Dun_VACo} for details---then shows that the Lie rank of $V_1$ is bounded above by the central charge of $V$. Applying Proposition 5.14 in \cite{Dun_VACo} we conclude that $V_1$ is a semi-simple Lie algebra of dimension $276$, with Lie rank bounded above by $12$. 

At this point our argument has converged with that used to establish Theorem 5.15 in \cite{Dun_VACo}. Picking up at the second paragraph of the proof of Theorem 5.15 in \cite{Dun_VACo} we see that an application of the main result of \cite{MR2219226} shows that $V_1$ is the simple complex Lie algebra of type $D_{12}$, and that the vertex operators on $V$ equip $V$ with a module structure of level $1$ for the affine Lie algebra of type ${D}^{(1)}_{12}$. So $V_{\bar{0}}$ is isomorphic to the lattice vertex operator algebra attached to the $D_{12}$ lattice. Proceeding as in the proof of Proposition \ref{prop:moon:cnstn-vatwmodstruc} we see that $V$ itself is isomorphic to a lattice super vertex operator algebra, and the lattice must be that obtained by adjoining some $\lambda_x$ say, of (\ref{eqn:moon:cnstn-lambdax}), to $V$. Since $V_{1/2}=\{0\}$, either $\lambda_x=\lambda_s$ or $\lambda_x=\lambda_c$, but both choices define isomorphic lattices, and hence isomorphic super vertex operator algebras. This completes the proof.
\end{proof}

\subsection{Principal Moduli}\label{sec:moon:mtseries}

The spin group $\Spin(\a)$ acts naturally on $\vsn$ and $\vsnt$, respecting the super vertex algebra and canonically-twisted module structures, in the sense that (\ref{eqn:gps:spin-spinactsvops}) holds for $x\in\Spin(\a)$, $a,b\in \vsn$ and $c\in \vsnt$. In particular, the $\Spin(\a)$-actions preserve the gradings defined by $\Lo$. We may compute the associated graded traces explicitly, and will do so momentarily (cf. Lemma \ref{lem:moon:mtseries-gddtracefns}). 

Since $\vsn$ and $\vsnt$ are super spaces, it is natural to consider graded super traces. Recall that the {\em super trace} of a parity preserving operator $X$ on a super vector space $V=V_{\bar 0}\oplus V_{\bar 1}$ is defined by setting
\begin{gather} 
	\str_V X:=\tr_{V_{\bar 0}}X-\tr_{V_{\bar 1}}X.
\end{gather}
(By parity preserving we just mean $X(V_{\bar j})\subset V_{\bar j}$.)

Observe that the super space gradings on $A(\a)$ and $A(\a)_\tw$ are given by the eigenspace decompositions for $\zz$, so we have
\begin{gather}
	\str_{A(\a)}Xq^{L(0)-c/24}=\tr_{A(\a)}\zz X q^{L(0)-c/24},\\
	\str_{A(\a)_\tw}Xq^{L(0)-c/24}=\tr_{A(\a)_\tw}\zz X q^{L(0)-c/24},	
\end{gather}
for $X$ an operator on $A(\a)$, $A(\a)_\tw$, that commutes with $L(0)$ and $\zz$.

Given $g\in \SO(\a)$ define $\eta_g(\tau)$ by setting
\begin{gather}\label{eqn:moon:mtseries-defnetax}
	\eta_g(\tau):=q\prod_{i=1}^{24}\prod_{n>0}(1-\varepsilon_iq^n)
\end{gather}
where $q=e^{2\pi\ii \tau}$, and the $\varepsilon_i$ are the eigenvalues for the action of $g$ on $\a$. Given $x\in\Spin(\a)$ write $C_x$ for the super trace of $x$ (i.e. the ordinary trace of $\zz x$) as an operator on $\CM$.
\begin{gather}\label{eqn:moon:mtseries-defnCx}
	C_x:=\str_{\CM}x=\tr_{\CM}\zz x
\end{gather}

A simple calculation now reveals that the graded super traces of $x\in\Spin(\a)$ on $A(\a)$ and $A(\a)_\tw$ are given by
\begin{gather}
	\str_{A(\a)}xq^{\Lo-c/24}
	=\frac{\eta_{\overline{x}}(\tau/2)}{\eta_{\overline{x}}(\tau)}\label{eqn:moon:mtseries-trxA},\\
	\str_{A(\a)_\tw}xq^{\Lo-c/24}
	=C_x\eta_{\overline{x}}(\tau)\label{eqn:moon:mtseries-trxAtw},
\end{gather}
where $\overline{x}$ is a shorthand for $x(\cdot)$, being the image of $x$ in $\SO(\a)$. Note that $c=12$ since $\dim\a=24$. (Cf. (\ref{eqn:va:cm-vir}).) 

This leads us quickly to expressions for the graded super traces of an arbitrary $x\in \Spin(\a)$ on $\vsn$ and $\vsnt$, which we record in the following lemma.

\begin{lemma}\label{lem:moon:mtseries-gddtracefns}
For $x\in \Spin(\a)$ the graded super traces for the actions of $x$ on $\vsn$ and $\vsnt$ are given by 
\begin{gather}
	\str_{\vsn}xq^{\Lo-c/24}=
	\frac{1}{2}\left(
		\frac{\eta_{\overline{x}}(\tau/2)}{\eta_{\overline{x}}(\tau)}
		+\frac{\eta_{-\overline{x}}(\tau/2)}{\eta_{-\overline{x}}(\tau)}
		+
		C_x\eta_{\overline{x}}(\tau)
		-C_{\zz x}\eta_{-\overline{x}}(\tau)
		\right),\label{eqn:moon:mtseries-trxVU}
\\
	\str_{\vsnt}xq^{\Lo-c/24}=
	\frac{1}{2}\left(
		\frac{\eta_{\overline{x}}(\tau/2)}{\eta_{\overline{x}}(\tau)}
		-\frac{\eta_{-\overline{x}}(\tau/2)}{\eta_{-\overline{x}}(\tau)}
		+
		C_x\eta_{\overline{x}}(\tau)
		+C_{\zz x}\eta_{-\overline{x}}(\tau)
		\right).\label{eqn:moon:mtseries-trxVV}
\end{gather}
\end{lemma}

The construction of \S\ref{sec:moon:cnstn} equips $\vsn$ and $\vsnt$ with actions by a group $\coh<\Spin(\a)$ isomorphic to Conway's group, $\Co_0$. We recall our convention (cf. \S\ref{sec:gps:conway}) that any polarization $\a=\a^-\oplus\a^+$ used to realize $\vsn$ and $\vsnt$ is chosen so that the associated lift $\zz$ of $-1$ (cf. \S\ref{sec:gps:spin}) belongs to $\coh$.

Lemma \ref{lem:moon:mtseries-gddtracefns} now attaches two holomorphic functions on the upper half plane to each conjugacy class $[g]\subset\Co_0$; namely, the super traces defined by (\ref{eqn:moon:mtseries-trxVU}) and (\ref{eqn:moon:mtseries-trxVV}).
\begin{gather}
	T^s_g(\tau):=\label{eqn:moon:mtseries-Tsg}	
	\str_{\vsn}\wh{g}q^{\Lo-c/24}
\\
	T^s_{g,\tw}(\tau):=\label{eqn:moon:mtseries-Tsgtw}
	\str_{\vsnt}\wh{g}q^{\Lo-c/24}
\end{gather}

These functions $T^s_g$ and $T^s_{g,\tw}$ are special. In order to demonstrate this we first recall the following result due to Conway--Norton and Queen.
\begin{theorem}
[\cite{MR554399,MR628715}]\label{thm:moon:mtseries-CNQ}
For any $g\in \Co_0$, regarded as a subgroup of $\SO(\a)$,
the function
\begin{gather}\label{eqn:moon:mtseries-tg}
	t_g(\tau):=\frac{\eta_g(\tau)}{\eta_g(2\tau)}
\end{gather}
is a principal modulus for a genus zero group $\Gamma_g<\SL_2(\RR)$ containing some $\Gamma_0(N)$.
\end{theorem}
We recall that a subgroup $\Gamma<\SL_2(\RR)$ commensurable with $\SL_2(\ZZ)$ is called {\em genus zero} if the natural Riemann surface structure on the quotient $\Gamma\backslash\HH\cup\hat{\QQ}$ has genus zero, where $\hat{\QQ}=\QQ\cup\{\infty\}$. If $\Gamma$ has genus zero then the field of $\Gamma$-invariant meromorphic functions on $\HH$, with possible poles at the cusps $\Gamma\backslash\hat{\QQ}$, is a simple transcendental extension of $\CC$. A generator is called a {\em principal modulus} for $\Gamma$. (The term Hauptmodul is also commonly used for this.) 

Note that a complete description of the invariance groups of the $t_g$ appears in \cite{MR780666}. The reader may see also the tables in \S\ref{sec:mmdata}.

Comparing with (\ref{eqn:moon:mtseries-trxA}) we see that the principal moduli $t_g$ of Conway--Norton and Queen are recovered in a simple way from the action of $\coh$ on $A(\a)$. Namely, 
\begin{gather}\label{eqn:moon:mtseries-tgfroma}
	t_{g}(\tau)
	=
	\str_{A(\a)}\wh{g}q^{2\Lo-c/12}
	=
	\tr_{A(\a)}\zz\wh{g}q^{2\Lo-c/12}
\end{gather}
for $g\in\coh\simeq \Co_0$. So the trace functions obtained from the action of $\Co_0$ on $A(\a)$ are principal moduli, according to Theorem \ref{thm:moon:mtseries-CNQ}.

We will show that the functions $T^s_g$, defined in (\ref{eqn:moon:mtseries-trxVU}) by the action of $\Co_0$ on the super vertex operator algebra $\vsn$, are also principal moduli for all $g\in\coh\simeq \Co_0$, but are distinguished in that they also satisfy the normalization condition 
\begin{gather}\label{eqn:moon:mtseries-npm}
T^s_g(2\tau)=q^{-1}+O(q). 
\end{gather}
Note that this condition does not hold in general for $t_{g}$, for we have $t_g(\tau)=q^{-1}-\chi_g+O(q)$, where $\chi_g$ denotes the character of $\coa\simeq \Co_0$ defined by its action on $\a=\LL\otimes_{\ZZ}\CC$  (cf. (\ref{eqn:gps:conway-chig})). The discrete subgroups of $\SL_2(\RR)$ attached to $\Co_0$ via the $T^s_g$ are essentially the same as those arising from the $t_g$ of Conway--Norton and Queen: it will develop that $T^s_g(2\tau)=t_g(\tau)+\chi_g$ for all $g\in \Co_0$. (Cf. Theorem \ref{thm:moon:mtseries-mainthm1}.) 

For explicit computations with $T^s_g$ and $T^s_{g,\tw}$ the notion of Frame shape is useful. Since $\coa\simeq\Co_0$ is the automorphism group of an integral lattice in $\a$ (cf. \S\ref{sec:gps:conway}) the traces $\chi_g$ for $g\in \coa$ are all integers. So the characteristic polynomial for the action of $g\in \coa$ on $\a$ can be written in the form $\prod_{m>0}(1-x^m)^{k_m}$ for some non-negative integers $k_m$ (all but finitely many being zero). In this situation then we find $\eta_g(\tau)=\prod_{m>0}\eta(m\tau)^{k_m}$ upon comparing with (\ref{eqn:moon:mtseries-defnetax}). 

The formal product 
\begin{gather}\label{eqn:frameshape}
\pi_g:=\prod_{m>0}m^{k_m}
\end{gather}
is called the {\em Frame shape} of $g$. We may define $\eta_{\pi}(\tau)$ for an arbitrary formal product $\pi=\prod_{m>0}m^{k_m}$ (with all but finitely many $k_m$ equal to zero) by setting 
\begin{gather}	
	\eta_{\pi}(\tau):=\prod_{m>0}\eta(m\tau)^{k_m}. 
\end{gather}
Of course then $\eta_g=\eta_{\pi_g}$ for $g\in \coa$. 

Thus $T^s_g$ and $T^s_{g,\tw}$ can be expressed explicitly in terms of the data $\pi_{\pm g}$ and $C_{\wh{\pm g}}$. (Note that $\wh{-g}=\zz\wh{g}$.) This data is collected in the tables of \S\ref{sec:mmdata} for all $g\in \Co_0$. Note that the trace of $g$ as an operator on $\a$ can also be read off from the Frame shape of $g$, for if $\pi_g=\prod_{m>0}m^{k_m}$ then $\chi_g=k_1$.

The values $C_{\wh{g}}$ are determined up to a sign by the eigenvalues of $g$. Indeed, suppose $g\in \coa\simeq \Co_0$ and assume, as in the discussion immediately following Lemma \ref{lem:moon:mtseries-gddtracefns}, that $\a=\a^-\oplus \a^+$ is a polarization such that $\a^{\pm}$ is spanned by (isotropic) eigenvectors $a^{\pm}_i$ for $g$, constituting a pair of dual bases in the sense that $\lab a^-_i,a^+_j\rab=\delta_{i,j}$. Assume also, as usual, that the associated lift $\zz$ of $-\Id_\a$ belongs to our chosen copy $\coh$ of $\Co_0$ in $\Spin(\a)$. Write $\lambda^{\pm 1}_i$ for the eigenvalue of $g$ attached to $a^{\pm}_i$. Then, after choosing $\alpha_i\in 2\pi \QQ$ such that $\lambda_i^{\pm 1}=e^{\pm 2\alpha_i \ii}$, we see that the product $x=\prod_{i=1}^{12}e^{\alpha_i X_i}$ is a lift of $g$ to $\Spin(\a)$, where the $X_i\in\g<\Cliff(\a)$ are defined as in (\ref{eqn:gps:spin-constzz}). Setting $\nu_i=e^{\alpha_i \ii}$ we obtain that the trace of $x$ on $\CM$ is given by $\prod_{i=1}^{12}(\nu_i+\nu_i^{-1})$, or equivalently, by $\nu\prod_{i=1}^{12}(1+\lambda_i^{- 1})$ where $\nu=\prod_{i=1}^{12} \nu_i$ is one of the two square roots of $\prod_{i=1}^{12} \lambda_i$. 

We conclude from this that $C_{\wh{g}}$ is given by
\begin{gather}\label{eqn:moon:mtseries-Cgeig}
	C_{\wh{g}}=\nu\prod_{i=1}^{12}(1-\lambda_i^{-1})
\end{gather}
(cf. (\ref{eqn:moon:mtseries-defnCx})) where $\nu$ is one of the two square roots of $\prod_{i=1}^{12}\lambda_i$. In particular, $C_{\wh{g}}=0$ if and only if $g$ has a fixed point in $\a$.

Our proof that the $T^s_g$ are normalized principal moduli depends upon the following lemma.

\begin{lemma}\label{lem:moon:mtseries-etaid}
For $g\in G\simeq\Co_0$ we have 
\be\label{eqn:moon:mtseries-etaid}
2\chi_g
-\frac{\eta_{-g}(\tau/2)}{\eta_{-g}(\tau)}+\frac{\eta_g(\tau/2)}{\eta_g(\tau)}+C_{\wh{-g}}\eta_{-g}(\tau)-C_{\wh{g}}\eta_g(\tau)=0.
\ee
\end{lemma}
\begin{proof}
Rewrite the required identity (\ref{eqn:moon:mtseries-etaid}) in the form
\begin{gather}\label{eqn:moon:mtseries-traceid}
\frac{1}{2}\left(\frac{\eta_{-g}(\tau/2)}{\eta_{-g}(\tau)}-\frac{\eta_{g}(\tau/2)}{\eta_{g}(\tau)}\right)
=
\chi_g+
\frac{1}{2}\left(C_{\wh{-g}}\eta_{-g}(\tau)-C_{\wh{g}}\eta_{g}(\tau)\right).
\end{gather}
Then, noting the identities (\ref{eqn:moon:mtseries-trxA}) and (\ref{eqn:moon:mtseries-trxAtw}), we recognize the left-hand side of (\ref{eqn:moon:mtseries-traceid}) as $\tr_{A(\a)^1}\hat{g}q^{L(0)-c/24}$, and the right-hand side as $\chi_g+\tr_{A(\a)_\tw^1}\hat{g}q^{L(0)-c/24}$. We now modify slightly the notational convention (\ref{eqn:va:funds-degdec}) to write 
\begin{gather}
A(\a)^1=\bigoplus_{n\in\ZZ}(A(\a)^1)_n
\end{gather}
for the grading of $A(\a)^1$ arising from the action of $L(0)-\frac{1}{2}\Id$, and similarly for $A(\a)^1_\tw$. Observe that both gradings are concentrated in non-negative degrees, positive in the case of $A(\a)^1_\tw$. Also, $(A(\a)^1)_0$ is isomorphic to $\a$ as a $\coh$-module, by construction. So we require to show that $(A(\a)^1)_n\simeq (A(\a)^1_\tw)_n$ as $\coh$-modules, for each positive integer $n$.

Recall (cf. Proposition \ref{prop:moon:cnstn-vatwmodstrucp}) that $\vfnt=A(\a)^1\oplus A(\a)_\tw^1$ may be regarded as the unique canonically-twisted module for $\vfn=A(\a)^0\oplus A(\a)_\tw^0$. Recall also (cf. Proposition \ref{prop:moon:cnstn-cotons}) that there is a uniquely determined $\coh$-invariant $N=1$ element $\tau\in \vfn$. Then the Fourier components of the twisted vertex operator $Y_{\tw}(\tau,z):\vfnt\to\vfnt\ldp z\rdp$ define an action of the Ramond algebra (cf. \S\ref{sec:va:funds}) on $\vfnt$. The defining relations show that $G(0):=\tau_{(1/2)}$ commutes with $L(0)$, and therefore preserves the subspaces $(A(\a)^1)_n\oplus (A(\a)_\tw^1)_n<\vfnt$. According to the fusion rules described in the proof of Proposition \ref{prop:moon:cnstn-vatwmodstruc}, the restriction of $G(0)$ to $A(\a)^1$ must map to $A(\a)_\tw^1$ (and vice versa). Now $G(0)^2=L(0)-\frac{1}{2}\Id$, so $G(0)$ defines an injective map $(A(\a)^1)_n\to (A(\a)_\tw^1)_n$ for all positive integers $n$. Since $\tau$ is $\coh$-invariant, these maps $(A(\a)^1)_n\to (A(\a)_\tw^1)_n$ are embeddings of $\coh$-modules. 

The lemma follows then if we can verify that $(A(\a)^1)_n$ and $(A(\a)^1_\tw)_n$ have the same dimension, for all $n>0$. That is, we should verify the $g=e$ case of (\ref{eqn:moon:mtseries-traceid}), which is the identity
\begin{gather}\label{eqn:moon:mtseries-ramdelid}
\frac{1}{2}\left(\frac{\Delta(\tau)^2}{\Delta(2\tau)\Delta(\tau/2)}-\frac{\Delta(\tau/2)}{\Delta(\tau)}\right)
=24+2^{11}\frac{\Delta(2\tau)}{\Delta(\tau)},
\end{gather}   
where $\Delta(\tau)=\eta(\tau)^{24}$ is the Ramanujan Delta function. This can be checked in a number of ways. For example, $f(\tau)=\Delta(2\tau)/\Delta(\tau)$ is a $\Gamma_0(2)$-invariant function, so the same is true of 
\begin{gather}
	(T_2f)(\tau):=\frac{1}{2}\left( f\left(\frac{\tau}{2}\right)+f\left(\frac{\tau+1}{2}\right)\right).
\end{gather}
(Cf. \cite[\S IX.6]{MR1193029}.) Now $f$ is actually a principal modulus for $\Gamma_0(2)$, with a simple pole at the unique non-infinite cusp, so $T_2f$ has a pole of order at most $2$ at the non-infinite cusp of $\Gamma_0(2)$, and no other poles. So $T_2f$ is a polynomial in $f$, of degree at most $2$, i.e.
\begin{gather}\label{eqn:moon:mtseries-almostalmostramdelid}
	T_2f=af^2+bf+c
\end{gather}
for some $a$, $b$ and $c$. Inspecting the first four coefficients of $f$ we see that $a=2048=2^{11}$, $b=24$ and $c=0$. Observing that $f(\frac{\tau+1}{2})=-f(\tau)/f(\tau/2)$, we now obtain 
\begin{gather}\label{eqn:moon:mtseries-almostramdelid}
	\frac{1}{2}\left(f\left(\frac{\tau}{2}\right)-\frac{f(\tau)}{f(\frac{\tau}{2})}\right)=24f(\tau)+2^{11}f(\tau)^2
\end{gather}
from (\ref{eqn:moon:mtseries-almostalmostramdelid}), and (\ref{eqn:moon:mtseries-ramdelid}) follows upon division of (\ref{eqn:moon:mtseries-almostramdelid}) by $f(\tau)$. The proof of the lemma is complete.
\end{proof}

We now come to the main results of this paper.

\begin{theorem}\label{thm:moon:mtseries-mainthm1}
Let $g\in \Co_0$. Then $T^s_g$ is the normalized principal modulus for a genus zero subgroup of $\SL_2(\RR)$. 
\end{theorem}
\begin{theorem}\label{thm:moon:mtseries-mainthm2}
Let $g\in \Co_0$. Then $T^s_{g,\tw}$ is constant, with constant value $-\chi_g$, when $g$ has a fixed point in its action on the Leech lattice. If $g$ has no fixed points then $T^s_{g,\tw}$ is a principal modulus for a genus zero subgroup of $\SL_2(\RR)$.
\end{theorem}
It is convenient to prove Theorems \ref{thm:moon:mtseries-mainthm1} and \ref{thm:moon:mtseries-mainthm2} together.
\begin{proof}[Proof of Theorems \ref{thm:moon:mtseries-mainthm1} and \ref{thm:moon:mtseries-mainthm2}.]
With $\coh$ a lift of $\coa\simeq\Co_0$ to $\Spin(\a)$ as before, define 
\begin{gather}
\tilde{t}_g(\tau):=t_g(\tau/2)=\str_{A(\a)}\wh{g}q^{L(0)-c/24}
\end{gather}
for $g\in \coa$ (cf. (\ref{eqn:moon:mtseries-tg}) and (\ref{eqn:moon:mtseries-tgfroma})), and define also the twisted analogues, 
\begin{gather}
	\tilde{t}_{g,\tw}(\tau):=\str_{A(\a)_\tw}\wh{g}q^{L(0)-c/24}. 
\end{gather}
Then $t_{g,\tw}(\tau)=C_{\wh{g}}\eta_{g}(\tau)$ according to (\ref{eqn:moon:mtseries-trxAtw}), and so $t_{g,\tw}$ vanishes identically if and only if $g$ has a fixed point for its action on $\a=\LL\otimes_{\ZZ}\CC$ according to (\ref{eqn:moon:mtseries-Cgeig}). 

Using (\ref{eqn:moon:mtseries-trxVU}) and (\ref{eqn:moon:mtseries-trxVV}) we may now write
\begin{gather}
	T^s_g=\frac{1}{2}\left(\tilde{t}_g+\tilde{t}_{-g}+\tilde{t}_{g,\tw}-\tilde{t}_{-g,\tw}\right),\label{eqn:moon:mtseries-Tgtildes}\\
	T^s_{g,\tw}=\frac{1}{2}\left(\tilde{t}_g-\tilde{t}_{-g}+\tilde{t}_{g,\tw}+\tilde{t}_{-g,\tw}\right),\label{eqn:moon:mtseries-Tgtwtildes}
\end{gather}
and the identity (\ref{eqn:moon:mtseries-etaid}) may be rewritten $\chi_g+\frac{1}{2}(\tilde{t}_{g}-\tilde{t}_{-g}-\tilde{t}_{g,\tw}+\tilde{t}_{-g,\tw})=0$. So applying (\ref{eqn:moon:mtseries-etaid}), to (\ref{eqn:moon:mtseries-Tgtildes}) and (\ref{eqn:moon:mtseries-Tgtwtildes}), we obtain 
\begin{gather}
	T^s_g=\tilde{t}_g+\chi_g,\label{eqn:moon:mtseries-Tgtildetgchi}\\
	T^s_{g,\tw}=\tilde{t}_{g,\tw}-\chi_g.\label{eqn:moon:mtseries-Tgtwtildetgtwchi}
\end{gather}
Since ${t}_g=q^{-1}-\chi_g+O(q)$ by inspection, (\ref{eqn:moon:mtseries-Tgtildetgchi}) verifies that $T^s_g(2\tau)=q^{-1}+O(q)$ for $g\in \coa$, and so $T^s_g(2\tau)$ is a normalised principal modulus according to Theorem \ref{thm:moon:mtseries-CNQ}. This proves Theorem \ref{thm:moon:mtseries-mainthm1}. 

The equation (\ref{eqn:moon:mtseries-Tgtwtildetgtwchi}) verifies that $T^s_{g,\tw}$ is constant, with constant value $-\chi_g$, when $g$ has a fixed point for its action on the Leech lattice according to the first paragraph of this proof, so it remains to understand $T^s_{g,\tw}$ in the case that $g$ has no fixed points. 

Observe that $C_{\wh{g}}/\tilde{t}_{g,\tw}=1/\eta_g$. If $g$ has no fixed points then the Frame shape $\pi_g=\prod_{m>0}m^{k_m}$ satisfies $\sum_{m>0}k_m=0$, so $1/\eta_g$ is a modular function for some congruence subgroup of $\SL_2(\ZZ)$. In fact, it has been verified in \cite{MR554399} that $1/\eta_g$ is, up to an additive constant, the McKay--Thompson series of an element of the monster group, for every such $g$ in the Conway group. So $T^s_{g,\tw}=-\chi_g+C_{\wh{g}}\eta_g$ is indeed a principal modulus for a genus zero subgroup of $\SL_2(\RR)$, whenever $g$ has no fixed points in $\LL$. The monster elements corresponding to elements $g\in \Co_0$ without fixed points may be read off from Table \ref{tab:mmdata-Tsgtw} in \S\ref{sec:mmdata}. This completes the proof of Theorem \ref{thm:moon:mtseries-mainthm2}.
\end{proof}

\section*{Acknowledgement}
We thank Miranda Cheng, Xi Dong, Igor Frenkel, Matthias Gaberdiel, Terry Gannon, Sarah Harrison, Jeff Harvey, Shamit Kachru and Timm Wrase for discussions on related topics. We also thank Jeff Harvey for correcting an error in an earlier draft.
The first author gratefully acknowledges support from the U.S. National Science Foundation (DMS 1203162).

\appendix


\clearpage

\section{Data}\label{sec:mmdata}

In Tables \ref{tab:mmdata-Tsg} and \ref{tab:mmdata-Tsgtw} we give all the data necessary for explicit computation of the McKay--Thompson series $T^s_g$ (cf. (\ref{eqn:moon:mtseries-Tsg})) and $T^s_{g,\tw}$ (cf. (\ref{eqn:moon:mtseries-Tsgtw})), attached to the Conway group $\Co_0$ in this paper.

In both tables, the first column lists the conjugacy class of the element $g\in\Co_0$ under consideration, and the next two columns list this element's associated $\Co_1$ class (cf. \S\ref{sec:gps:conway}) and Frame shape $\pi_g$ (cf. \eqref{eqn:frameshape}). These are followed by the super trace $C_{\widehat{-g}}$ of $\widehat{-g}$ on $\CM$ (in Table \ref{tab:mmdata-Tsg}) or the super trace $C_{\widehat{g}}$ of $\widehat{g}$ on $\CM$ (in Table \ref{tab:mmdata-Tsgtw}; cf. \eqref{eqn:moon:mtseries-defnCx}). In the fifth column we describe explicitly the invariance groups of the McKay--Thompson series, writing $\Gamma_g$ for the invariance group of $T^s_g$ (Table \ref{tab:mmdata-Tsg}), and $\Gamma_{g,\tw}$ for the invariance group of $T^s_{g,\tw}$ (Table \ref{tab:mmdata-Tsgtw}). The invariance groups in Table \ref{tab:mmdata-Tsgtw} are each associated to an element of the monster group by monstrous moonshine, and this is listed in the last column labelled $\MM$.

Note that a complete description of the groups $\Gamma_g$ first appeared in \cite{MR780666}, but our notation, in Table \ref{tab:mmdata-Tsg}, is different in certain cases, adhering more closely to the traditions initiated in \cite{MR554399}. More specifically, we follow the conventions of \cite{MR1291027}, so that $n|h-$, for example, when $h$ is the largest divisor of $24$ such that $h^2$ divides $nh$, denotes the subgroup of index $h$ in $\Gamma_0(n/h)$ defined in \cite{MR554399}. (See also \cite{Fer_Genus0prob} for a detailed analysis of the groups $n|h-$, and their extensions by Atkin--Lehner involutions.) So, for example, $12+3$ denotes the group obtained by adjoining an Atkin--Lehner involution $W_3=\frac{1}{\sqrt{3}}\left(\begin{smallmatrix} 3a&b\\12c&3d\end{smallmatrix}\right)$ to $\Gamma_0(12)$, where $9ad-12bc=3$.

Not all the groups $\Gamma_g$ appear in \cite{MR1291027}, so we need some additional notation. We use $\up\tfrac1h$ and $n\dn$ to denote upper and lower triangular matrices, respectively, 
\begin{gather}
\up\tfrac1h :=\left(\begin{matrix} 1&\tfrac1h\\0&1\end{matrix}\right),\quad
n\dn :=\left(\begin{matrix} 1&0\\n&1\end{matrix}\right).
\end{gather}

We then write $12+3\up\tfrac12$, for example, for the group generated by $\Gamma_0(12)$ and the product of $W_3$ with $\up\tfrac12$, where $W_3$ is an Atkin--Lehner involution for $\Gamma_0(12)$, as in the previous paragraph. Note that this group is also denoted $12+3'$ in \cite{MR1291027}. Now the group denoted $4|2-$ in \cite{MR554399,MR1291027} can be described as $8+\up\tfrac124\dn$, for it is generated by $\Gamma_0(8)$ together with the product of $\up\tfrac12$ and $4\dn$. For $8|2+$ we may write $16+16,\up\tfrac128\dn$, meaning the group generated by $\Gamma_0(16)$, the Fricke involution $\frac{1}{4}\left(\begin{smallmatrix} 0&-1\\16&0\end{smallmatrix}\right)$, and the product of $\up\tfrac12$ with $8\dn$.

Note that $\Gamma_g$ and $\Gamma_{-g}$ are related by conjugation by $\up\tfrac12$, for every $g\in \Co_0$, 
\begin{gather}
\Gamma_{-g}=\Gamma_g^{\up\tfrac12}=
\left(\begin{matrix} 1&-\tfrac12\\0&1\end{matrix}\right)
\Gamma_g
\left(\begin{matrix} 1&\tfrac12\\0&1\end{matrix}\right),
\end{gather}
since the Fourier expansions of the functions $T^s_g(2\tau)$ and $T^s_{-g}(2\tau)$ differ exactly by signs on even powers of $q$.

As mentioned in \S\ref{sec:moon:mtseries}, the invariance groups $\Gamma_{g,\tw}$, of the canonically-twisted McKay--Thompson series $T^s_{g,\tw}$, are all genus zero groups that arise in monstrous moonshine. We include the corresponding monstrous class names in Table \ref{tab:mmdata-Tsgtw}, where the $\Gamma_{g,\tw}$ are described explicitly.

\clearpage

\begin{center}
\begingroup
\renewcommand{\arraystretch}{1.5}
\begin{longtable}{cccrl}
\caption{Data for the $T^s_g$}\label{tab:mmdata-Tsg} \\

\multicolumn{1}{c}{$\Co_0$} & 
\multicolumn{1}{c}{$\Co_1$} & 
\multicolumn{1}{c}{$\pi_g$} & 
\multicolumn{1}{r}{$C_{\wh{-g}}$} & 
\multicolumn{1}{l}{$\Gamma_g$} \\ \hline\hline
\endfirsthead

\multicolumn{5}{c}%
{{\tablename\ \thetable{}, continued from previous page}} \\
\multicolumn{1}{c}{$\Co_0$} & 
\multicolumn{1}{c}{$\Co_1$} & 
\multicolumn{1}{c}{$\pi_g$} & 
\multicolumn{1}{r}{$C_{\wh{-g}}$} & 
\multicolumn{1}{l}{$\Gamma_g$} \\ \hline\hline
\endhead

\hline \multicolumn{5}{r}{{Continued on next page}} \\
\endfoot

\hline
\endlastfoot

1A & 1A & $1^{24}$ & 4096 & $2-$ \\
2A & 1A & $2^{24}/1^{24}$ & 0 & $4+$ \\
\hline
2B & 2A & $1^8 2^8$ & 0 & $4-$ \\
2C & 2A & $2^{16}/1^8$ & 0 & $4-$ \\
\hline
4A & 2B & $4^{12}/2^{12}$ & 64 & $8|2+$\\
\hline
2D & 2C & $2^{12}$ & 0 & $4|2-$ \\
\hline
3A & 3A & $3^{12}/1^{12}$ & 1 & $6+6$ \\
6A & 3A & $1^{12}6^{12}/2^{12}3^{12}$ & 729 & $(6+6)^{\up\tfrac12}$\\
\hline
3B & 3B & $1^63^6$ & 64 & $6+3$ \\
6B & 3B & $2^66^6/1^63^6$ & 0 & $12+$ \\
\hline
3C & 3C & $3^9/1^3$ & $-8$ & $6-$ \\
6C & 3C & $1^36^9/2^33^9$ & 0 & $12+4$ \\
\hline
3D & 3D & $3^8$ & 16 & $6|3$ \\
6D & 3D & $6^8/3^8$ & 0 & $12|3+$ \\
\hline
4B & 4A & $1^84^8/2^8$ & 256 & $(8+)^{\up\tfrac12}$ \\
4C & 4A & $4^8/1^8$ & 0 & $8+$ \\
\hline
4D & 4B & $4^8/2^4$ & 0 & $8-$ \\
\hline
4E & 4C & $1^4 2^2 4^4$ & 0 & $8-$ \\
4F & 4C & $2^6 4^4/1^4$ & 0 & $8-$ \\
\hline
4G & 4D & $2^44^4$ & 0 & $8|2-$ \\
\hline
8A & 4E & $8^6/4^6$ & 8 & $16| 4+$ \\
\hline
4H & 4F & $4^6$ & 0 & $8|4-$ \\
\hline
5A & 5A & $5^6/1^6$ & 1 & $10+10$ \\
10A & 5A & $1^610^6/2^65^6$ & 125 & $\left(10+10\right)^{\up\tfrac12}$ \\
\hline
5B & 5B & $1^45^4$ & 16 & $10+5$ \\
10B & 5B & $2^410^4/1^45^4$ & 0 & $20+$ \\
\hline
5C & 5C & $5^5/1^1$ & $-4$ & $10-$ \\
10C & 5C & $1^110^5/2^15^5$ & 0 & $20+4$ \\
\hline
6E & 6A & $3^46^4/1^42^4$ & 9 & $12+12$ \\
6F & 6A & $1^46^8/2^83^4$ & 81 & $(12+12)^{\up\tfrac12}$ \\
\hline
12A & 6B & $2^612^6/4^66^6$ & 1 & $(12\vert2+6)^{\up\tfrac14}$ \\
\hline
6G & 6C & $2^5 3^4 6^1 /1^4$ & 0 & $12+3\up\tfrac12$ \\
6H & 6C & $1^4 2^1 6^5/3^4$ & 0 & $12+3\up\tfrac12$ \\
\hline
6I & 6D & $1^53^16^4/2^4$ & 72 & $(12+12)^{\up\tfrac12}$ \\
6J & 6D & $2^16^5/1^53^1$ & 0 & $12+12$ \\
\hline
6K & 6E & $1^2 2^2 3^2 6^2$ & 0 & $12+3$ \\
6L & 6E & $2^4 6^4/1^2 3^2$ & 0 & $12+3$ \\
\hline
6M & 6F & $3^3 6^3/1^1 2^1$ & 0 & $12-$ \\
6N & 6F & $1^16^6/2^23^3$ & 0 & $12-$ \\
\hline
6O & 6G & $2^36^3$ & 0 & $12\vert2+3\up\tfrac12$ \\
\hline
12B & 6H & $12^4/6^4$ & 4 & $24\vert6+$ \\
\hline
6P & 6I & $6^4$ & 0 & $12|6-$ \\
\hline
7A & 7A & $7^4/1^4$ & 1 & $14+14$ \\
14A & 7A & $1^414^4/2^47^4$ & 49 & $\left(14+14\right)^{\up\tfrac12}$ \\
\hline
7B & 7B & $1^37^3$ & 8 & $14+7$ \\
14B & 7B & $2^314^3/1^37^3$ & 0 & $28+$ \\
\hline
8B & 8A & $8^4/2^4$ & 16 & $16\vert2+$ \\
\hline
8C & 8B & $2^48^4/4^4$ & 0 & $(16\vert2+)^{\up\tfrac14}$ \\
\hline
8D & 8C & $1^48^4/2^24^2$ & 32 & $(16+)^{\up\tfrac12}$ \\
8E & 8C & $2^28^4/1^44^2$ & 0 & $16+$ \\
\hline
8F & 8D & $8^4/4^2$ & 0 & $16-$ \\
\hline
8G & 8E & $1^2 2^1 4^1 8^2$ & 0 & $16-$ \\
8H & 8E & $2^34^1 8^2/1^2$ & 0 & $16-$ \\
\hline
8I & 8F & $4^28^2$ & 0 & $16\vert4-$ \\
\hline
9A & 9A & $9^3/1^3$ & 1 & $18+18$ \\
18A & 9A & $1^318^3/2^39^3$ & 27 & $\left(18+18\right)^{\up\tfrac12}$ \\
\hline
9B & 9B & $9^3/3^1$ & $-2$ & $18-$ \\
18B & 9B & $3^118^3/6^19^3$ & 0 & $36+4$ \\
\hline
9C & 9C & $1^39^3/3^2$ & 4 & $18+9$ \\
18C & 9C & $2^33^218^3/1^36^29^3$ & 0 & $36+$ \\
\hline
10D & 10A & $5^210^2/1^22^2$ & 5 & $20+20$ \\
10E & 10A & $1^210^4/2^45^2$ & 25 & $(20+20)^{\up\tfrac12}$ \\
\hline
20A & 10B & $2^320^3/4^310^3$ & $-1$ & $(20\vert2+10)^{\up\tfrac14}$ \\
\hline
20B & 10C & $4^220^2/2^210^2$ & 4 & $40\vert2+$ \\
\hline
10F & 10D & $2^3 5^2 10^1 /1^2$ & 0 & $20+5\up\tfrac12$ \\
10G & 10D & $1^2 2^1 10^3/5^2$ & 0 & $20+5\up\tfrac12$ \\
\hline
10H & 10E & $1^35^110^2/2^2$ & 20 & $(20+20)^{\up\tfrac12}$ \\
10I & 10E & $2^110^3/1^35^1$ & 0 & $20+20$ \\
\hline
10J & 10F & $2^210^2$ & 0 & $20\vert2+5$ \\
\hline
11A & 11A & $1^211^2$ & 4 & $22+11$ \\
22A & 11A & $2^222^2/1^211^2$ & 0 & $44+$ \\
\hline
12C & 12A & $2^43^412^4/1^44^46^4$ & 1 & $24+24,3\up\tfrac1212\dn$ \\
12D & 12A & $1^412^4/3^44^4$ & 9 & $24+8,3\up\tfrac1212\dn$ \\
\hline
12E & 12B & $2^212^4/4^46^2$ & $-3$ & $12-$ \\
\hline
12F & 12C & $6^212^2/2^24^2$ & 9 & $24\vert2+12$ \\
\hline
12G & 12D & $2^13^312^3/1^14^16^3$ & 4 & $(24+8)^{\up\tfrac12}$ \\
12H & 12D & $1^112^3/3^34^1$ & 0 & $24+8$ \\
\hline
12I & 12E & $1^23^24^212^2/2^26^2$ & 16 & $(24+)^{\up\tfrac12}$ \\
12J & 12E & $4^212^2/1^23^2$ & 0 & $24+$ \\
\hline
24A & 12F & $4^324^3/8^312^3$ & $-1$ & $(24\vert4+6)^{\up\tfrac18}$ \\
\hline
12K & 12G & $4^212^2/2^16^1$ & 0 & $24+3\up\tfrac12$ \\
\hline
12L & 12H & $1^1 2^2 3^1 12^2/4^2$ & 0 & $(24\vert2+12)^{\up\tfrac14}$ \\
12M & 12H & $2^36^112^2/1^13^14^2$ & 0 & $(24\vert2+12)^{\up\tfrac14}$ \\
\hline
12N & 12I & $2^2 3^2 4^1 12^1/1^2$ & 0 & $24+3\up\tfrac12$ \\
12O & 12I & $1^2 4^1 6^2 12^1/3^2$ & 0 & $24+3\up\tfrac12$ \\
\hline
12P & 12J & $2^14^16^112^1$ & 0 & $24\vert2+3$ \\
\hline
12Q & 12K & $1^312^3/2^13^14^16^1$ & 12 & $(24+24)^{\up\tfrac12}$ \\
12R & 12K & $2^23^112^3/1^34^16^2$ & 0 & $24+24$ \\
\hline
24B & 12L & $24^2/12^2$ & 2 & $48\vert12+$ \\
\hline
12S & 12M & $12^2$ & 0 & $24|12-$\\
\hline
13A & 13A & $13^2/1^2$ & 1 & $26+26$ \\
26A & 13A & $1^226^2/2^213^2$ & 13 & $\left(26+26\right)^{\up\tfrac12}$ \\
\hline
28A & 14A & $2^228^2/4^214^2$ & 1 & $(28\vert2+14)^{\up\tfrac14}$ \\
\hline
14C & 14B & $1^1 2^1 7^1 14^1$ & 0 & $28+7$ \\
14D & 14B & $2^214^2/1^17^1$ & 0 & $28+7$ \\
\hline
15A & 15A & $1^315^3/3^35^3$ & 1 & $30+6,10,15$ \\
30A & 15A & $2^33^35^330^3/1^36^310^315^3$ & $-1$ & $\left(30+6,10,15\right)^{\up\tfrac12}$ \\
\hline
15B & 15B & $3^215^2/1^25^2$ & 1 & $30+5,6,30$ \\
30B & 15B & $1^25^26^230^2/2^23^210^215^2$ & 9 & $\left(30+5,6,30\right)^{\up\tfrac12}$ \\
\hline
15C & 15C & $15^2/3^2$ & 1 & $30\vert3+10$ \\
30C & 15C & $3^230^2/6^215^2$ & 5 & $(30\vert3+10)^{\up\tfrac12}$ \\
\hline
15D & 15D & $1^13^15^115^1$ & 4 & $30+3,5,15$ \\
30D & 15D & $2^16^110^130^1/1^13^15^115^1$ & 0 & $60+$ \\
\hline
15E & 15E & $1^215^2/3^15^1$ & 2 & $30+15$ \\
30E & 15E & $2^23^15^130^2/1^26^110^115^2$ & 0 & $\left(30+15\right)^{\up\tfrac12}$ \\
\hline
16A & 16A & $2^216^2/4^18^1$ & 0 & $(32\vert2 +)^{\up\tfrac14}$ \\
\hline
16B & 16B & $1^216^2/2^18^1$ & 8 & $(32+)^{\up\tfrac12}$ \\
16C & 16B & $2^116^2/1^28^1$ & 0 & $32+$ \\
\hline
18D & 18A & $9^118^1/1^12^1$ & 3 & $36+36$ \\
18E & 18A & $1^118^2/2^29^1$ & 9 & $(36+36)^{\up\tfrac12}$ \\
\hline
18F & 18B & $1^29^118^1/2^13^1$ & 6 & $(36+36)^{\up\tfrac12}$ \\
18G & 18B & $2^13^118^2/1^26^19^1$ & 0 & $36+36$ \\
\hline
18H & 18C & $2^29^118^1/1^16^1$ & 0 & $36+9\up\tfrac12$ \\
18I & 18C & $1^12^118^2/6^19^1$ & 0 & $36+9\up\tfrac12$ \\
\hline
20C & 20A & $2^25^220^2/1^24^210^2$ & 1 & $40+8,5\up\tfrac12 20\dn$ \\
20D & 20A & $1^220^2/4^25^2$ & 5 & $40+40,5\up\tfrac12 20\dn$ \\
\hline
20E & 20B & $4^120^1$ & 0 & $40\vert4+5\up\tfrac12$ \\
\hline
20F & 20C & $2^25^120^1/1^14^1$ & 0 & $(40\vert2+20)^{\up\tfrac14}$ \\
20G & 20C & $1^12^110^120^1/4^15^1$ & 0 & $(40\vert2+20)^{\up\tfrac14}$ \\
\hline
21A & 21A & $1^221^2/3^27^2$ & 1 & $42+6,14,21$ \\
42A & 21A & $2^23^27^242^2/1^26^214^221^2$ & 1 & $\left(42+6,14,21\right)^{\up\tfrac12}$ \\
\hline
21B & 21B & $7^121^1/1^13^1$ & 1 & $42+3,14,42$ \\
42B & 21B & $1^13^114^142^1/2^16^17^121^1$ & 7 & $\left(42+3,14,42\right)^{\up\tfrac12}$ \\
\hline
21C & 21C & $3^121^1$ & 2 & $42\vert3+7$ \\
42C & 21C & $6^142^1/3^121^1$ & 0 & $(42\vert3+7)^{\up\tfrac12}$ \\
\hline
22BC & 22A & $2^122^1$ & 0 & $44\vert2+11\up\tfrac12$ \\
\hline
23AB & 23AB & $1^123^1$ & 2 & $46+23$ \\
46AB & 23AB & $2^146^1/1^123^1$ & 0 & $92+$ \\
\hline
24C & 24A & $2^224^2/6^28^2$ & 1 & $96+32,96^{24\dn},\up\tfrac12 48\dn$ \\
\hline
24D & 24B & $2^13^24^124^2/1^26^18^212^1$ & $-1$ & $48+16,48^{\up\tfrac12}$ \\
24E & 24B & $1^24^16^124^2/2^13^28^212^1$ & 3 & $48+48,16^{\up\tfrac12}$ \\
\hline
24F & 24C & $8^124^1/2^16^1$ & 4 & $48\vert2+$ \\
\hline
24G & 24D & $12^124^1/4^18^1$ & 3 & $48\vert4+12$ \\
\hline
24H & 24E & $2^16^18^124^1/4^112^1$ & 0 & $(48\vert2+)^{\up\tfrac14}$ \\
\hline
24I & 24F & $2^13^14^124^1/1^18^1$ & 0 & $(48\vert4+12)^{\up\tfrac18}$ \\
24J & 24F & $1^14^16^124^1/3^18^1$ & 0 & $(48\vert4+12)^{\up\tfrac18}$ \\
\hline
52A & 26A & $2^152^1/4^126^1$ & $-1$ & $(52\vert2 + 26)^{\up\tfrac14}$ \\
\hline
28B & 28A & $1^14^17^128^1/2^114^1$ & 4 & $(56+)^{\up\tfrac12}$ \\
28C & 28A & $4^128^1/1^17^1$ & 0 & $56+$ \\
\hline
56AB & 28B & $4^156^1/8^128^1$ & 1 & $(56\vert4+14)^{\up\tfrac18}$ \\
\hline
30F & 30A & $1^12^115^130^1/3^15^16^110^1$ & $-1$ & $60+12,15,20$ \\
30G & 30A & $2^23^15^130^2/1^16^210^215^1$ & 1 & $(60+12,15,20)^{\up\tfrac12}$ \\
\hline
60A & 30B & $2^110^112^160^1/4^16^120^130^1$ & 1 & $(60\vert2+5,6,30)^{\up\tfrac14}$ \\
\hline
60B & 30C & $6^160^1/12^130^1$ & $-1$ & $(60\vert6+10)^{\up\tfrac{1}{12}}$ \\
\hline
30H & 30D & $1^16^110^115^1/3^15^1$ & 0 & $60+3\up\tfrac12,5\up\tfrac12,15$ \\
30I & 30D & $2^13^15^130^1/1^115^1$ & 0 & $60+3\up\tfrac12,5\up\tfrac12,15$ \\
\hline
30J & 30E & $2^13^15^130^1/6^110^1$ & 2 & $60+12,15,20$ \\
30K & 30E & $2^130^1/3^15^1$ & 0 & $60+12,15,20$ \\
\hline
33A & 33A & $3^133^1/1^111^1$ & 1 & $66+6,11,66$ \\
66A & 33A & $1^16^111^166^1/2^13^122^133^1$ & 3 & $\left(66+6,11,66\right)^{\up\tfrac12}$ \\
\hline
35A & 35A & $1^135^1/5^17^1$ & 1 & $70+10,14,35$ \\
70A & 35A & $2^15^17^170^1/1^110^114^135^1$ & $-1$ & $\left(70+10,14,35\right)^{\up\tfrac12}$ \\
\hline
36A & 36A & $2^19^136^1/1^14^118^1$ & 1 & $72+8,9\up\tfrac12 36\dn$ \\
36B & 36A & $1^136^1/4^19^1$ & 3 & $72+72,9\up\tfrac12 36\dn$ \\
\hline
39AB & 39AB & $1^139^1/3^113^1$ & 1 & $78+6,26,39$ \\
78AB & 39AB & $2^13^113^178^1/1^16^126^139^1$ & 1 & $\left(78+6,26,39\right)^{\up\tfrac12}$ \\
\hline
40AB & 40A & $2^140^1/8^110^1$ & 1 & $160+32,160^{40\dn},\up\tfrac12 40\dn$ \\
\hline
84A & 42A & $4^16^114^184^1/2^112^128^142^1$ & 1 & $(84\vert2+6,14,21)^{\up\tfrac14}$ \\
\hline
60C & 60A & $1^14^16^110^115^160^1/2^13^15^112^120^130^1$ & 1 & $120+15,24,3\up\tfrac12 60\dn$ \\
60D & 60A & $3^14^15^160^1/1^112^115^120^1$ & $-1$ & $120+15,120,3\up\tfrac12 60\dn$ \\

\end{longtable}
\end{center}
\endgroup

\newpage

\begin{center}
\begingroup
\renewcommand{\arraystretch}{1.5}
\begin{longtable}{cccrlc}
\caption{Data for the $T^s_{g,\mathrm{tw}}$}\label{tab:mmdata-Tsgtw}\\

\multicolumn{1}{c}{$\Co_0$} & 
\multicolumn{1}{c}{$\Co_1$} & 
\multicolumn{1}{c}{$\pi_g$} & 
\multicolumn{1}{r}{$C_{\wh{g}}$} & 
\multicolumn{1}{l}{$\Gamma_{g,\tw}$} & 
\multicolumn{1}{c}{$\MM$} \\ \hline \hline
\endfirsthead

\multicolumn{6}{c}%
{{\tablename\ \thetable{}, continued from previous page}} \\
\multicolumn{1}{c}{$\Co_0$} & 
\multicolumn{1}{c}{$\Co_1$} & 
\multicolumn{1}{c}{$\pi_g$} & 
\multicolumn{1}{r}{$C_{\wh{g}}$} & 
\multicolumn{1}{l}{$\Gamma_{g,\tw}$} & 
\multicolumn{1}{c}{$\MM$} \\ \hline\hline
\endhead

\hline \multicolumn{6}{r}{{Continued on next page}} \\
\endfoot

\hline
\endlastfoot

2A & 1A & $2^{24}/1^{24}$ & 4096 & $2-$ & 2B \\
\hline
4A & 2B & $4^{12}/2^{12}$ & 64 & $4\vert 2-$ & 4D \\
\hline
3A & 3A & $3^{12}/1^{12}$ & 729 & $3-$ & 3B \\
6A & 3A & $1^{12}6^{12}/2^{12}3^{12}$ & 1 & $6+6$ & 6B \\
\hline
6B & 3B & $2^66^6/1^63^6$ & 64 & $6+3$ & 6C \\
\hline
6C & 3C & $1^36^9/2^33^9$ & $-8$ & $6-$ & 6E \\
\hline
6D & 3D & $6^8/3^8$ & 16 & $6\vert 3-$ & 6F \\
\hline
4C & 4A & $4^8/1^8$ & 256 & $4-$ & 4C \\
\hline
8A & 4E & $8^6/4^6$ & 8 & $8\vert 4-$ & 8F \\
\hline
5A & 5A & $5^6/1^6$ & 125 & $5-$ & 5B \\
10A & 5A & $1^610^6/2^65^6$ & 1 & $10+10$ & 10D \\
\hline
10B & 5B & $2^410^4/1^45^4$ & 16 & $10+5$ & 10B \\
\hline
10C & 5C & $1^110^5/2^15^5$ & $-4$ & $10-$ & 10E \\
\hline
6E & 6A & $3^46^4/1^42^4$ & 81 & $6+2$ & 6D \\
6F & 6A & $1^46^8/2^83^4$ & 9 & $6-$ & 6E \\
\hline
12A & 6B & $2^612^6/4^66^6$ & 1 & $12\vert 2+6$ & 12F \\
\hline
6J & 6D & $2^16^5/1^53^1$ & 72 & $6-$ & 6E \\
\hline
12B & 6H & $12^4/6^4$ & 4 & $12\vert 6-$ & 12J \\
\hline
7A & 7A & $7^4/1^4$ & 49 & $7-$ & 7B \\
14A & 7A & $1^414^4/2^47^4$ & 1 & $14+14$ & 14C \\
\hline
14B & 7B & $2^314^3/1^37^3$ & 8 & $14+7$ & 14B \\
\hline
8B & 8A & $8^4/2^4$ & 16 & $8\vert 2-$ & 8D \\
\hline
8E & 8C & $2^28^4/1^44^2$ & 32 & $8-$ & 8E \\
\hline
9A & 9A & $9^3/1^3$ & 27 & $9-$ & 9B \\
18A & 9A & $1^318^3/2^39^3$ & 1 & $18+18$ & 18E \\
\hline
18B & 9B & $3^118^3/6^19^3$ & $-2$ & $18-$ & 18D \\
\hline
18C & 9C & $2^33^218^3/1^36^29^3$ & 4 & $18+9$ & 18C \\
\hline
10D & 10A & $5^210^2/1^22^2$ & 25 & $10+2$ & 10C \\
10E & 10A & $1^210^4/2^45^2$ & 5 & $10-$ & 10E \\
\hline
20A & 10B & $2^320^3/4^310^3$ & $-1$ & $20\vert 2+10$ & 20E \\
\hline
20B & 10C & $4^220^2/2^210^2$ & 4 & $20\vert 2+5$ & 20D \\
\hline
10I & 10E & $2^110^3/1^35^1$ & 20 & $10-$ & 10E \\
\hline
22A & 11A & $2^222^2/1^211^2$ & 4 & $22+11$ & 22B \\
\hline
12C & 12A & $2^43^412^4/1^44^46^4$ & 9 & $12+4$ & 12B \\
12D & 12A & $1^412^4/3^44^4$ & 1 & $12+12$ & 12H \\
\hline
12E & 12B & $2^212^4/4^46^2$ & $-3$ & $12-$ & 12I \\
\hline
12F & 12C & $6^212^2/2^24^2$ & 9 & $12\vert 2+2$ & 12G \\
\hline
12H & 12D & $1^112^3/3^34^1$ & 4 & $12-$ & 12I \\
\hline
12J & 12E & $4^212^2/1^23^2$ & 16 & $12+3$ & 12E \\
\hline
24A & 12F & $4^324^3/8^312^3$ & $-1$ & $24\vert 4+6$ & 24F \\
\hline
12R & 12K & $2^23^112^3/1^34^16^2$ & 12 & $12-$ & 12I \\
\hline
24B & 12L & $24^2/12^2$ & 2 & $24\vert 12-$ & 24J \\
\hline
13A & 13A & $13^2/1^2$ & 13 & $13-$ & 13B \\
26A & 13A & $1^226^2/2^213^2$ & 1 & $26+26$ & 26B \\
\hline
28A & 14A & $2^228^2/4^214^2$ & 1 & $28\vert 2+14$ & 28D \\
\hline
15A & 15A & $1^315^3/3^35^3$ & $-1$ & $15+15$ & 15C \\
30A & 15A & $2^33^35^330^3/1^36^310^315^3$ & 1 & $30+6,10,15$ & 30A \\
\hline
15B & 15B & $3^215^2/1^25^2$ & 9 & $15+5$ & 15B \\
30B & 15B & $1^25^26^230^2/2^23^210^215^2$ & 1 & $30+5,6,30$ & 30D \\
\hline
15C & 15C & $15^2/3^2$ & 5 & $15\vert 3-$ & 15D \\
30C & 15C & $3^230^2/6^215^2$ & 1 & $30\vert 3+10$ & 30E \\
\hline
30D & 15D & $2^16^110^130^1/1^13^15^115^1$ & 4 & $30+3,5,15$ & 30C \\
\hline
30E & 15E & $2^23^15^130^2/1^26^110^115^2$ & 2 & $30+15$ & 30G \\
\hline
16C & 16B & $2^116^2/1^28^1$ & 8 & $16-$ & 16B \\
\hline
18D & 18A & $9^118^1/1^12^1$ & 9 & $18+2$ & 18A \\
18E & 18A & $1^118^2/2^29^1$ & 3 & $18-$ & 18D \\
\hline
18G & 18B & $2^13^118^2/1^26^19^1$ & 6 & $18-$ & 18D \\
\hline
20C & 20A & $2^25^220^2/1^24^210^2$ & 5 & $20+4$ & 20C \\
20D & 20A & $1^220^2/4^25^2$ & 1 & $20+20$ & 20F \\
\hline
21A & 21A & $1^221^2/3^27^2$ & 1 & $21+21$ & 21D \\
42A & 21A & $2^23^27^242^2/1^26^214^221^2$ & 1 & $42+6,14,21$ & 42B \\
\hline
21B & 21B & $7^121^1/1^13^1$ & 7 & $21+3$ & 21B \\
42B & 21B & $1^13^114^142^1/2^16^17^121^1$ & 1 & $42+3,14,42$ & 42D \\
\hline
42C & 21C & $6^142^1/3^121^1$ & 2 & $42\vert 3+7$ & 42C \\
\hline
46AB & 23AB & $2^146^1/1^123^1$ & 2 & $46+23$ & 46AB \\
\hline
24C & 24A & $2^224^2/6^28^2$ & 1 & $24\vert 2+12$ & 24H \\
\hline
24D & 24B & $2^13^24^124^2/1^26^18^212^1$ & 3 & $24+8$ & 24C \\
24E & 24B & $1^24^16^124^2/2^13^28^212^1$ & $-1$ & $24+24$ & 24I \\
\hline
24F & 24C & $8^124^1/2^16^1$ & 4 & $24\vert 2+3$ & 24D \\
\hline
24G & 24D & $12^124^1/4^18^1$ & 3 & $24\vert 4+2$ & 24G \\
\hline
52A & 26A & $2^152^1/4^126^1$ & $-1$ & $52\vert 2+26$ & 52B \\
\hline
28C & 28A & $4^128^1/1^17^1$ & 4 & $28+7$ & 28C \\
\hline
56AB & 28B & $4^156^1/8^128^1$ & 1 & $56\vert 4+14$ & 56BC \\
\hline
30F & 30A & $1^12^115^130^1/3^15^16^110^1$ & 1 & $30+2,15,30$ & 30F \\
30G & 30A & $2^23^15^130^2/1^16^210^215^1$ & $-1$ & $30+15$ & 30G \\
\hline
60A & 30B & $2^110^112^160^1/4^16^120^130^1$ & 1 & $60\vert 2+5,6,30$ & 60E \\
\hline
60B & 30C & $6^160^1/12^130^1$ & $-1$ & $60\vert 6+10$ & 60F \\
\hline
30K & 30E & $2^130^1/3^15^1$ & 2 & $30+15$ & 30G \\
\hline
33A & 33A & $3^133^1/1^111^1$ & 3 & $33+11$ & 33A \\
66A & 33A & $1^16^111^166^1/2^13^122^133^1$ & 1 & $66+6,11,66$ & 66B \\
\hline
35A & 35A & $1^135^1/5^17^1$ & $-1$ & $35+35$ & 35B \\
70A & 35A & $2^15^17^170^1/1^110^114^135^1$ & 1 & $70+10,14,35$ & 70B \\
\hline
36A & 36A & $2^19^136^1/1^14^118^1$ & 3 & $36+4$ & 36B \\
36B & 36A & $1^136^1/4^19^1$ & 1 & $36+36$ & 36D \\
\hline
39AB & 39AB & $1^139^1/3^113^1$ & 1 & $39+39$ & 39CD \\
78AB & 39AB & $2^13^113^178^1/1^16^126^139^1$ & 1 & $78+6,26,39$ & 78BC \\
\hline
40AB & 40A & $2^140^1/8^110^1$ & 1 & $40\vert 2+20$ & 40CD \\
\hline
84A & 42A & $4^16^114^184^1/2^112^128^142^1$ & 1 & $84\vert 2+6,14,21$ & 84B \\
\hline
60C & 60A & $1^14^16^110^115^160^1/2^13^15^112^120^130^1$ & $-1$ & $60+4,15,60$ & 60C \\
60D & 60A & $3^14^15^160^1/1^112^115^120^1$ & 1 & $60+12,15,20$ & 60D \\

\end{longtable}
\end{center}
\endgroup

\clearpage


\end{document}